\documentclass[bj,authoryear]{imsart}

\usepackage{natbib}
\usepackage{amsmath, amsthm, amsfonts, mathrsfs, amsfonts, amssymb}
\usepackage[final]{microtype}
\usepackage{enumitem}
\usepackage{graphicx}
\usepackage{calligra}
\DeclareMathAlphabet{\mathcalligra}{T1}{calligra}{m}{n}
\usepackage{hyperref}

\startlocaldefs

\theoremstyle{plain}
\newtheorem{theorem}{Theorem}
\newtheorem{lemma}{Lemma}
\newtheorem{proposition}{Proposition}
\newtheorem{corollary}{Corollary}

\theoremstyle{remark}
\newtheorem{assumption}{Assumption}
\newtheorem{remark}{Remark}

\newcommand{\ind}{\stackrel{\text{ind}}{\sim}}
\def\generalprob#1{\ {\mathchoice{\raise-1.5pt\hbox{$\buildrel#1\over\ra$}}
		{\raise-2pt\hbox{$\buildrel#1\over\ra$}}{}{}}\ }

\def\l{\lambda}

\newcommand{\given}{\mid}

\newcommand\NN{\mathbb{N}}
\newcommand\g{\gamma}
\newcommand\e{\epsilon}
\newcommand\h{\eta}

\newcommand\expt{\E}
\DeclareMathOperator{\E}{E}
\newcommand\trans{^{\intercal}}

\def\t{\tau}

\def\q{\theta}

\def\d{\delta}
\def\O{\mathcal{O}}

\def\bi{i}
\def\mcalu{{u}}
\def\mcall{{l}}
\def\weak{\rightsquigarrow}

\newcommand\RR{\mathbb{R}}
\newcommand\iid{\stackrel{\text{i.i.d.}}{\sim}}

\newcommand\ra{\rightarrow}

\newcommand\tr{\mathop{\rm tr}\nolimits}
\newcommand\argmin{\mathop{\rm arg min}}
\def\vec{\mathop{\text{vec}}\nolimits}
\newcommand\var{\mathop{\mathrm{Var}}\nolimits}
\newcommand\Cov{\mathop{\mathrm{Cov}}\nolimits}

\newcommand{\N}{\ensuremath{\mathbb{N}}}

\newcommand{\m}{\ensuremath{\mu}}

\endlocaldefs

\begin{document}

\begin{frontmatter}
\title{Misspecified Bernstein-Von Mises theorem for hierarchical models}
\runtitle{Hierarchical misspecified Bernstein-Von Mises}

\begin{aug}
\author[A]{\inits{G.}\fnms{Geerten}~\snm{Koers}\ead[label=e1]{g.j.koers@tudelft.nl}}
\author[B]{\inits{B.}\fnms{Botond}~\snm{Szabó}\ead[label=e2]{botond.szabo@unibocconi.it}}
\author[A]{\inits{A}\fnms{Aad}~\snm{van der Vaart}\ead[label=e3]{a.w.vandervaart@tudelft.nl}}
\address[A]{DIAM, Delft University of Technology, Mekelweg 4, 2628 CD, Delft, Netherlands, \printead[presep={,\ }]{e1,e3}}

\address[B]{Department of Data Sciences and BIDSA, Bocconi University, Via Roentgen 1 20136, Milan, Italy, \printead{e2}}
\end{aug}

\begin{abstract}
We derive a Bernstein Von-Mises theorem in the context of misspecified, non-i.i.d., hierarchical models parametrised by a finite-dimensional parameter of interest.
We apply our results to hierarchical models containing non-linear operators, including the squared integral operator, and PDE-constrained inverse problems.
More specifically, we consider the elliptic, time-independent Schrödinger equation with parametric boundary condition and general parabolic PDEs with parametric potential and boundary constraints.
Our theoretical results are complemented with a numerical analysis of synthetic data sets, considering both the square integral operator and the Schrödinger equation. 
\end{abstract}


\end{frontmatter}


\section{Introduction}\label{sec:Introduction}
Hierarchical models are widely used across various fields, such as astronomy (weak lensing \cite{alsing2016hierarchical,sellentin2018skewed}, cosmic microwave background temperature maps \cite{eriksen2004power}, gravitational waves \cite{cornish2015bayeswave}), statistical physics (the Dyson hierarchical model \cite{monthus2011critical}), environmental sciences (predicting the spread of ecological processes \cite{wikle2003:eco}) and medicine (spatial modelling of fMRI data \cite{bowman2008bayesian}). These models define a data-generating process in multiple layers, offering great flexibility in modeling. The connections between the layers are often non-linear and governed by partial differential equations (PDEs), which introduces significant analytical and computational challenges.

The objective is to infer aspects of the hierarchical process. Directly applying exact methods for inference is usually computationally prohibitive or infeasible. Therefore, surrogate or approximate likelihoods are often used instead of the exact likelihood of the multilayer data-generating structure. Misspecified Gaussian likelihoods are popular for their computational convenience and ease of interpretation. However, this simplification results in information loss and may lead to inaccurate error estimates. Despite the known inaccuracies of these misspecified methods, they are often used without rigorous theoretical guarantees, sometimes leading to contradictory results.

This article was motivated by a specific example from astronomy. Cosmic microwave background radiation, a snapshot of the universe approximately $3\times 10^5$  years after the Big Bang, is modeled as a Gaussian random field (although its Gaussianity is debated \cite{marinucci2004testing}). Present-day observations are non-linear transformations of this field, described by complex systems of PDEs, corrupted with noise. Due to the non-linearity, the observed data are non-Gaussian. The goal is to recover key parametric aspects of the non-linear operators (and their surrogates), including the proportions of dark energy and dark matter, and the Hubble constant (the universe's expansion rate). The article \cite{ezquiaga2018dark} reviews tools for understanding dark energy using gravitational waves. To simplify modeling and computations, the complex hierarchical model is replaced by a Gaussian approximation. Figure~8  of the article shows that the error bars of the proposed estimators do not intersect, leading to contradictory conclusions about the universe's expansion rate. Understanding the theoretical properties of misspecified hierarchical approaches is crucial to determine whether these non-overlapping confidence intervals result from using different models or statistical error. 

In this article we consider Bayesian methods for inferring parametric aspects of the hierarchical model. Bayesian methods are popular across many fields due to their built-in uncertainty quantification and their natural way of incorporating prior knowledge. Our focus is on the theoretical, asymptotic properties of the posterior distribution of the parameter of interest. For correctly specified, regular parametric models, the Bernstein-Von Mises theorem shows that the posterior distribution is asymptotically Gaussian, centred at the maximum likelihood (or another efficient) estimator, and has variance equal to the Cram\'er-Rao bound. Efficiency and accuracy of the uncertainty quantification can be derived from this result (\cite{doob1949application,le2012asymptotic, VanDerVaart1998}). Extensions to semi-parametric and non-parametric models were considered in \cite{Castillo2012,castillo2015bernstein,Leahu:2011,CastilloNickl2014}, while a more accurate, skewed version of the theorem was derived in \cite{skewed:BvM}.

As noted, the hierarchical models that motivate the present article are typically misspecified. For this situation
 Bernstein-Von Mises type results have been established for various models, see \cite{Kleijn12} and the review in \cite{review:misspecified}. 
However, none of these results specifically address hierarchical models, and the conditions do not apply to our non-linear examples. Thus, we first derive a new, modified version of the classical misspecified theorem. Next, we apply this to a range of hierarchical models, starting with the square integral operator as a toy example and next turning to partial differential equation (PDE) constrained inverse problems. We first examine the time-independent Schrödinger equation, which allows tractable computations and serves as a foundation for exploring more complex PDE-constrained inverse problems. We then derive asymptotic guarantees for general parabolic PDEs with parametric potential and boundary constraints, considering observations at different time points to mimic typical non-i.i.d.\ observational structures in physics and astronomy.

The article is organised as follows.
In Section~\ref{sec:Description of the problem} we introduce the hierarchical observational model, its Gaussian surrogate and the misspecified Bayesian framework.
In Section~\ref{sec:BvM} we provide our general misspecified Bernstein-Von Mises theorem and in Section~\ref{sec:Applications} we apply it to hierarchical models.
Our main contribution is to cover several interesting, PDE-constrained inverse problems, and is presented in Section~\ref{sec:examples}.
In Section~\ref{sec:Simulations} we investigate the numerical behaviour of the limiting misspecified posterior in contrast to the well-specified case and the accuracy of the Gaussian approximation.
The proofs of the general results, examples, and additional technical lemmas are deferred to Sections~\ref{sec:Proofs},~\ref{sec:proof:examples},~\ref{sec:Technical Lemmas}, and~\ref{sec:proof:BvM:direct}.

\section{Description of the problem}\label{sec:Description of the problem}
We consider non-linear, hierarchical models of the form
\begin{equation}
\begin{aligned}\label{eq: model}
f_i &\iid G, \qquad\qquad &i = 1,\ldots, N,\\
X_i \given f_i, \q &\ind N_p(T_{\q}^{\bi}(f_i), \Lambda^{\bi}), \qquad &i = 1, \ldots, N.
\end{aligned}
\end{equation}
Here $G$ is a probability distribution on some function space $\mathcal{F}$ equipped with its Borel sets,
$T_{\q}^{\bi}:\mathcal{F}\to \RR^p$ are linear or non-linear operators, indexed by an unknown finite-dimensional parameter 
$\q\in\RR^d$, and $\Lambda^{\bi}\in \RR^{p\times p}$ are known, positive-definite covariance matrices.
We denote the (marginal, not conditional on $f_i$) distribution of $X_i$ by $P_{\q,i}$, for $i = 1,\ldots, N$, 
with density $p_{\q,i}$,  and denote the corresponding expectation and covariance matrix by
\begin{equation}
\label{EqMeanCovarianceModel}
\mu_{\q}^{\bi} := \expt_{\q,i} X_i,\qquad\qquad\Sigma_{\q}^{\bi}:= \Cov_{\q, i}(X_i),\qquad i = 1,\ldots, N.
\end{equation}
The objective is to make inference on the unknown model parameter $\q\in\Theta\subset\RR^d$.
In case the operators $T_{\q}^{\bi}$ are the same for all $i=1,\ldots, N$, we arrive at an i.i.d.\ model for $(X_i)_{i = 1, \ldots, N}$.

The first layer in \eqref{eq: model} models a random, non-parametric, unobserved functional parameter $f_i$.
The second layer describes the actual observations $X_1,\ldots, X_N$, which depend on a (possibly) non-linear transformation of the functional parameter, 
and are corrupted with Gaussian noise. Their distributions $P_{\q,i}$ are mixtures of multivariate-normal distributions.
In Section~\ref{sec:examples} we consider several specific examples of this form, including elliptic and parabolic PDE-constrained inverse problems.
In applications the operators $T_{\q}^{\bi}$ can be complex (see the examples in the introduction), but the main goal is typically the same: to recover certain parametric aspects of the model/operator.

In practice, simplified, approximate models are used to speed up computations and increase the interpretability of the model.
In particular, Gaussian approximations are common. In our analysis we replace the mixture likelihood $P_{\q,i}$ of $X_i$ by a Gaussian approximation 
with mean and covariance matching the corresponding quantities in the original model, given by
\begin{equation}
\label{EqMisspecifiedQ}
Q_{\q,i}=N_p(\mu_{\q}^{\bi},\Sigma_{\q}^{\bi}).
\end{equation}
We assume that the mean and covariance functions \eqref{EqMeanCovarianceModel} can be computed,
or at least can be suitably approximated. This may itself be computationally intensive, but less forbidding and more 
accessible to modelling than computing the full likelihood function. 

Using a misspecified model results in a  loss of information and may invalidate the usual methods of uncertainty quantification.
Below we investigate these aspects for the Bayesian approach, in a frequentist setting, where we assume that the data are in
reality generated according to the hierarchical model \eqref{eq: model}.

Given a prior density $\pi$ on the parameter set $\Theta$, the misspecified posterior distribution is given by, for $q_{\q,i}$ a density of $Q_{\q,i}$,
\begin{align}\label{def:posterior}
\Pi(\q \in \cdot \given X_1, \ldots, X_N) = \frac{ \int_\cdot \prod_{i=1}^N q_{\q, i}(X_i)\pi(\q) \,d\q}{\int_{\Theta} \prod_{i=1}^N q_{\q, i}(X_i) \pi(\q) \,d\q}.
\end{align}
We derive a Bernstein-Von Mises type result, which shows that these random measures approximate 
to a Gaussian distribution as $N\ra\infty$.
Under mild assumptions these distributions contract around the parameter $\q_N^* \in \Theta$ 
that minimizes the Kullback-Leibler divergence between the misspecified Gaussian class and the original model, i.e.
\begin{equation*}
\q_N^*=\argmin_{\q\in\Theta } P_{\q_{0}}^{(N)}\biggl( \sum_{i=1}^N \log \frac{p_{\q_{0},i}(X_i)}{q_{\q, i}(X_i)}\biggr).
\end{equation*}
We show that in our setup $\q_N^*$ coincides with the parameter of interest $\q_{0}$, 
derive the covariance matrix of the limiting Gaussian distribution, and
investigate how this differs from the well-specified posterior distribution, 
both theoretically and numerically for various synthetic data sets.

\section{Misspecified Bernstein-Von Mises theorem}\label{sec:BvM}
The asymptotic behavior of the posterior distribution in regular parametric models is described by the well-known Bernstein-Von Mises theorem. In our case, we construct the posterior distribution using the likelihood of the Gaussian model \eqref{EqMisspecifiedQ} as a surrogate for the hierarchical model \eqref{eq: model}.

In this section, we first derive a Bernstein-Von Mises theorem for misspecified models that accommodates this scenario. The proposition is stated in general terms, extending beyond our specific hierarchical model \eqref{eq: model}.

\subsection{A general Bernstein-Von Mises theorem}
Misspecified Bernstein-Von Mises results were derived in the literature before, 
see for instance \cite{Kleijn12}, but under conditions too strong for our situation.
In particular, to obtain a $\sqrt{N}$-contraction rate, the aforementioned paper assumes that the log misspecified likelihood is locally Lipschitz with a majoring function that has exponential moments, which is violated for examples of our model \eqref{eq: model}, as demonstrated in Remark~\ref{remark:counter} in Section~\ref{sec:examples} below.
In this section, we present a theorem that is appropriate for the hierarchical model \eqref{eq: model} with a selection of operators $T_{\q}^i$.

The conditions of the theorem are essentially the classical ones given in Theorem~7.1 in \cite{Lehmann83} 
(given there for the well-specified case, i.i.d.\ data and a one-dimensional parameter), 
but are at a level of abstraction that they apply to general models, also beyond independent observations.

Let $l^{(N)}(\q)$ and $\nabla l^{(N)}(\q)$ denote a possibly misspecified log-likelihood and its gradient, of a model for an observation $X^{(N)}$.
Given $\q_N^*\in\RR^d$ and symmetric nonnegative-definite $d\times d$ matrices $V_{\q_N^*, N}$, define, for $\q\in\Theta$, symmetric matrices $R_N(\q)$ through
\begin{align}\label{eq:definition Rn}
l^{(N)}(\q) &= l^{(N)}(\q_N^*) + \nabla l^{(N)}(\q_N^*)\trans (\q - \q_N^*) \\
&\qquad\qquad- \frac{1}{2}N(\q - \q_N^*)\trans \bigl[V_{\q_N^*,N} + \frac1NR_N(\q)\bigr](\q - \q_N^*). \nonumber
\end{align}
Let $P_0^{(N)}$ denote the ``true'' distribution of $X^{(N)}$, which needs to have no structural relation to the likelihood \eqref{eq:definition Rn}. 
The misspecified posterior distribution based on the likelihood  \eqref{eq:definition Rn} will be investigated under $P_0^{(N)}$. In this investigation the vector 
$\q_N^*$ will typically minimise the map $\q\mapsto P_{0}^{(N)}l^{(N)}(\q)$,
and $V_{\q_N^*, N}$ will typically be the Hessian of the map $\q \mapsto -N^{-1} P_{0}^{(N)} l^{(N)}(\q)$ at the point $\q = \q_N^*$, but this is not necessary
and only the following assumptions are imposed. 

Since the dimension of $\q$ is fixed, the norm $\|\cdot\|$ in the following may be any vector or matrix norm.

\begin{assumption}\label{assum:BvM theorem}\hfill
\begin{enumerate}[label=\textbf{\ref*{assum:BvM theorem}.\arabic*},ref={\ref*{assum:BvM theorem}.\arabic*}]
\item \label{assum: convergence thetan} 
The sequence $\q_N^*$ tends to an interior point $\q^*$ of \space$\Theta \subset \RR^d$.
\item \label{assum: posit def V} $\|V_{\q_N^*, N} - V_{*}\| \ra 0$ for a positive-definite $d\times d$ matrix $V_{*}$ (typically dependent on $\theta^*$).
\item \label{assum: regular likel} $\frac{1}{\sqrt{N}}\nabla l^{(N)}(\q_N^*)=O_P(1)$ in $P_{0}^{(N)}$-probability as $n\to\infty$.
\item \label{assum:likel max} For any $\delta > 0$, there exists an $\epsilon > 0$ such that
\begin{align*}
\lim_{N\to\infty} P_{0}^{(N)} \biggl( \sup_{\q:\|\q - \q_N^*\|\geq \delta} \frac{1}{N}\bigl( l^{(N)}(\q) - l^{(N)}(\q_N^*)\bigr) \leq -\epsilon\biggr) = 1.
\end{align*}
\item \label{assum:Rn conv} Given any $\epsilon > 0$, there exists a $\delta > 0$ such that
\begin{align*}
\lim_{N\to\infty} P_{0}^{(N)}\biggl( \sup_{\q:\|\q - \q_N^*\| \leq \delta} \frac{1}{N}\bigl\| R_N(\q)\bigr\| \geq \epsilon \biggr) = 0.
\end{align*}
\item \label{assum:prior cont} The prior has a density $\pi$ that is continuous and positive at $\q=\q^*$.
\item \label{assum:prior exp} The prior has a finite $k$th moment: $\int_{\Theta}\|\q\|^k \pi(\q)\, d\q < \infty$.
\end{enumerate}
\end{assumption}

\begin{proposition}\label{thm:BvM direct proof}
Let $t\mapsto \pi_N(t \given  X^{(N)})$ be the density of $\sqrt{N}(\vartheta - T_N)$ given $X^{(N)}$ if $\vartheta$ follows the
 distribution with density proportional to $\q\mapsto \exp\bigl(l^{(N)}(\q)\bigr)\,\pi(\q)$, where
\begin{align}\label{def:Tn}
T_N = \q_N^* + \frac{1}{N} V_{\q_N^*, N}^{-1} \nabla l^{(N)}(\q_N^*).
\end{align}
Under Assumptions~\ref{assum: convergence thetan}--\ref{assum:prior cont},
\begin{align}\label{eq:BvM-convergence}
\int\limits_{\sqrt{N}(\Theta-T_N)}\!\!\!\!\!\!\! \bigl| \pi_N(t \given X^{(N)}) - \varphi_{0,V_{*}^{-1}}(t) \bigr|\, dt \stackrel{P_{0}^{(N)}}{\to} 0,
\end{align}
where $\varphi_{\mu,\Sigma}$ denotes the density of the $d$-dimensional normal distribution with mean vector $\mu\in\RR^d$ and covariance matrix $\Sigma\in\RR^{d\times d}$.
If in addition Assumption~\ref{assum:prior exp} holds, then also
\begin{align}\label{eq:BvM-convergence exp}
\int\limits_{\sqrt{N}(\Theta-T_N)}\!\!\!\!\!\!\! \bigl(1+\|t\|^k\bigr)\,\bigl| \pi_N(t \given X^{(N)}) - \varphi_{0,V_{*}^{-1}}(t) \bigr|\, dt \stackrel{P_{0}^{(N)}}{\to} 0.
\end{align}
\end{proposition}

The proof of the proposition is deferred to Section~\ref{sec:proof:BvM:direct}.

By the invariance of the total variation distance under measurable bijections, including affine transformations, the convergence \eqref{eq:BvM-convergence} implies that
the misspecified posterior distribution of $\vartheta$ approximates to a normal distribution with centre $T_N$ and covariance matrix $N^{-1}V_{*}^{-1}$.
However, in contrast to the well-specified case, the covariance matrix $V_{*}^{-1}$ does not necessarily match the limiting covariance matrix of 
the sequence $\sqrt N(T_N-\q_N^*)$. If Assumption~\ref{assum: regular likel}  is strengthened to the convergence in distribution 
$\frac{1}{\sqrt{N}}\nabla l^{(N)}(\q_N^*) \weak N_d(0, J_*)$, 
then the sequence $\sqrt N(T_N-\q_N^*)$ is asymptotically normal with mean zero and
covariance matrix $V_*^{-1}J_*V_*^{-1}$ of the ``sandwich'' form. The marginal variances $a^TV_{*}^{-1}a$, for $a\in\RR^d$, can be larger, equal, or smaller than the 
corresponding marginal variances of the latter matrix, see \cite{Kleijn12}. As a consequence,
credible sets resulting from the misspecified posterior distribution can be both overconfident or conservative, 
depending on the true distribution $P_{0}^{(N)}$, and can deviate from the credible sets resulting from a correctly specified model.

The Gaussian approximation to the posterior distribution obtainable from \eqref{eq:BvM-convergence} 
depends on the unknown parameter $\q_N^*$ through the centring at $T_N$ and the covariance matrix $V_*$. Because by \eqref{eq:BvM-convergence exp}
the mean and covariance matrix of the misspecified posterior distribution also converge (if $k\ge 2$), it follows that also the normal distribution
with these two data-driven parameters approximate to the misspecified posterior distribution.
Under regularity conditions (on the misspecified likelihood), the centring $T_N$ could also be replaced by 
the misspecified maximum likelihood estimator, the point of maximum of $\q\mapsto l^{(N)}(\q)$.

\subsection{Bernstein-Von Mises in hierarchical models}\label{sec:Applications}
In this section, we apply Proposition~\ref{thm:BvM direct proof} in the context of the hierarchical data-generating model \eqref{eq: model} and the misspecified surrogate Gaussian likelihood \eqref{EqMisspecifiedQ}.
Given observations $X_1,\ldots, X_N$, we form the posterior distribution in Equation~\eqref{def:posterior} with the $q_{\q, i}$ equal to the density of the normal distribution with the correctly specified means and covariance matrices, as in Equation~\eqref{EqMeanCovarianceModel}.
The true distribution $P_0^{(N)}$ is given by the hierarchical model specified by \eqref{eq: model}.

The precision matrix $V_{*}\in\RR^{d\times d}$ of the limiting Gaussian distribution takes a particular form.
The proof of the following lemma is deferred to Section~\ref{sec:prof:lem:variance}.

\begin{lemma}\label{thm:locMinKLDiv}
Consider the hierarchical model \eqref{eq: model} and suppose that 
$(\mu_{\q}^{\bi}, \Sigma_{\q}^{\bi}) \neq (\mu_{\q_{0}}^{\bi}, \Sigma_{\q_{0}}^{\bi})$, for all $\q \neq \q_{0}$.
Then the Kullback-Leiber divergence $\q \mapsto P_{\q_{0},i} \log(p_{\q_{0},i}/q_{\q,i})$ has a unique minimum at $\q = \q_{0}$.
Assume furthermore that the maps $\q\mapsto \mu_{\q}^{\bi}$ and $\q\mapsto\Sigma_\q^{\bi}$ are
twice continuously differentiable in a neighbourhood around $\q_{0}$,  and that $\Sigma_{\q_{0}}^{\bi}$ is invertible.
Then, with $v_l^{\bi} := (\Sigma_{\q_{0}}^{\bi})^{-1/2} \frac{d}{d\q_l} \mu_{\q}^{\bi}|_{\q = \q_{0}}$ and $A_l^{\bi} := (\Sigma_{\q_{0}}^{\bi})^{-1} \frac{d}{d\q_l} \Sigma_{\q}^{\bi}|_{\q=\q_{0}}$, $l=1,\ldots,d$, the Hessian at $\q=\q_{0}$ 
of the Kullback-Leibler divergence is given by the positive semi-definite matrix
\begin{align}\label{eq:Vstar}
V_{\q_{0}}^{\bi} = \frac{1}{2}
\begin{pmatrix}
\tr(A_1^{\bi} A_1^{\bi}) & \cdots & \tr(A_1^{\bi} A_d^{\bi})\\
\vdots & \ddots & \vdots\\
\tr(A_d^{\bi} A_1^{\bi}) & \cdots & \tr(A_d^{\bi} A_d^{\bi})
\end{pmatrix}
+ 
\begin{pmatrix}
(v_1^{\bi})\trans v_1^{\bi} & \cdots & (v_1^{\bi})\trans v_d^{\bi}\\
\vdots& \ddots & \vdots\\
(v_d^{\bi})\trans v_1^{\bi} & \cdots & (v_d^{\bi})\trans v_d^{\bi}
\end{pmatrix}.
\end{align}
\end{lemma}

The matrix $V_{\q_{0}}^{\bi}$ in Equation~\eqref{eq:Vstar} is the sum of two nonnegative-definite matrices 
and is strictly positive-definite as soon as one of these matrices is invertible.
The second matrix on the right is invertible if and only if the vectors $\frac{d}{d\q_l} \mu_{\q}^{\bi}|_{\q = \q_{0}}$, $l=1,\ldots, d$, are linearly independent.
The matrix $V_{\q_0}^{\bi}$ typically does not have a closed form analytic expression and hence numerical methods are used to evaluate it; see Section~\ref{sec:Simulations} for examples.

The following theorem specializes Proposition~\ref{thm:BvM direct proof} to the Gaussian misspecification of model \eqref{eq: model}.
The proof of the theorem consists of verifying the conditions of Proposition~\ref{thm:BvM direct proof}, 
and is given in Section~\ref{sec:proof:cor}.

\begin{theorem}\label{cor:BvM:hierarhical}
Consider the hierarchical model \eqref{eq: model} with $\q_0$ an interior point of the parameter set $\Theta\subset\RR^d$.
Assume that the mean and variance functions $\q \mapsto \mu_{\q}^{\bi}$ and $\q \mapsto \Sigma_{\q}^{\bi}$ are
twice continuously differentiable in a neighbourhood of $\q_{0}$, and that the matrices $\Sigma_{\q}^{\bi}$  are invertible.
Assume that $\sup_{i\in\NN}P_{\q_{0}}^{\bi} \|X_i\|^4 < \infty$ and that the prior density $\pi$ is continuous and positive  at $\q_{0}$.
Then \eqref{eq:BvM-convergence} of Proposition~\ref{thm:BvM direct proof}  holds with $\q^* = \q_{0}$ and $V_{*}=V_{\q_{0}}$
in both of the following cases:
\begin{enumerate}[label=(\roman*)]
\item[(a)] $\q \mapsto \mu_{\q}^{\bi}$ and $\q \mapsto \Sigma_{\q}^{\bi}$ are independent of $i$,
with  $\inf_{\|\q-\q_0\|\ge\d}\|\mu_{\q}^\bi-\m_{\q_0}^\bi\|+\|\Sigma_{\q}^\bi-\Sigma_{\q_{0}}^\bi\|>0$  for every $\d>0$, and
the matrix $V_{\q_0}:=V_{\q_{0}}^{\bi}$ given in Equation~\eqref{eq:Vstar} is invertible.
\item[(b)] $\q \mapsto \mu_{\q}^{\bi}$ and $\q \mapsto \Sigma_{\q}^{\bi}$ are equi-continuous on $\Theta$ ($i\in\NN$),
the maximum likelihood estimator of $\q$ is bounded in probability,
the limit $V_{\q_{0}}:=\lim_{N\to\infty} N^{-1}\sum_{i=1}^{N}V_{\q_{0}}^{\bi}$ exists and is invertible, and, 
for every $M>\delta > 0$,
\end{enumerate}
\newcommand\smSp{\mkern-0.75mu}
\begin{align}
\label{EqAsymptoticIdentifiability}
\limsup_{N\to\infty} \inf_{\q: M\ge \|\q-\q_{0}\|\ge \delta}\frac{1}{N} \sum_{i=1}^N \Bigl[&
\bigl(\mu_{\q}^{\bi} - \mu_{\q_{0}}^{\bi}\bigr)\trans \bigl(\Sigma_{\q}^{\bi}\bigr)^{-1} \bigl(\mu_{\q}^{\bi} - \mu_{\q_{0}}^{\bi}\bigr) \nonumber\\
&+ \smSp \tr\bigl(\Sigma_{\q_{0}}^{\bi}\bigl(\Sigma_{\q}^{\bi}\bigr)^{-1}-I\bigr) \smSp - \smSp \log\det\bigl(\Sigma_{\q_{0}}^{\bi}\bigl(\Sigma_{\q}^{\bi}\bigr)^{-1}\bigr)\Bigr] \smSp > 0 \smSp.\end{align}
\end{theorem}

The theorem ensures that the misspecified posterior distribution accumulates 
its mass within balls with a radius of the order $1/\sqrt{N}$ around the true parameter $\q_{0}$ and hence the posterior distribution is consistent at the optimal rate.
However, as noted following Theorem~\ref{thm:BvM direct proof}, credible sets do not necessarily coincide with confidence sets, not even asymptotically.
This is illustrated in our numerical analysis in Section~\ref{sec:Simulations}.

The left side of Equation~\eqref{EqAsymptoticIdentifiability} involves the average  Kullback-Leibler divergence 
between the surrogate Gaussian distributions \eqref{EqMisspecifiedQ}.
The condition ensures that the model is asymptotically identifiable. In case (a) of the theorem,
these averages are fixed in $N$ and condition \eqref{EqAsymptoticIdentifiability} is ensured by the condition
on the means $\m_\q^i$ and covariance matrices $\Sigma_\q^i$.

\section{Applications}\label{sec:examples}
In this section, we examine the misspecified posterior distribution based on the surrogate Gaussian likelihood \eqref{EqMisspecifiedQ}
in three versions of the hierarchical model \eqref{eq: model}. The first example is primarily of theoretic interest, whereas the other two examples are of more practical relevance,
with the  operator $T_{\q}^{\bi}$  is based on the forward map of a partial differential equation. 

In the first example, the functions $f_i$  in the initial step of the hierarchical model are Brownian motions, and the transformation
$T_{\q}^{\bi}(f_i)$ is based on the square $L_2$-norm of these functions shifted by a scalar parameter of interest $\q$.
By taking different limits on the integral, we can explore a non-i.i.d. model for the observations.

The second example involves the time-independent Schr\"odinger equation, where the parameter 
$\q$ specifies the boundary condition of the equation. We consider this example with i.i.d.\ data.

The third example deals with general parabolic PDEs, evaluated at different time points 
$t_i$, which creates a non-i.i.d. setting. This scenario is similar to situations in astronomy, 
where we observe snapshots of the universe at various stages of its evolution. 
The parameter of interest in this case specifies either the boundary condition or the potential function.

\subsection{Square integral operator}\label{subsec:square integral operator}
For $\q, z \in(0,\infty)$, define a map $\tau_{z, \q}: L^2[0,z] \to L^2[0,z]$ by
\begin{align}\label{def:exam1:T}
\tau_{z,\q}(h)(t) = \int_0^t (h(s) - \q)^2 \,ds,\qquad t\in[0,z].
\end{align}
We take the functional parameter $f$ in \eqref{eq: model} to be a Brownian motion on $[0,z]$,
and consider observing the projection of the function $\tau_{z, \q}(f)$ on the Legendre polynomials,
subject to Gaussian noise. The details are as follows.

The normalised Legendre polynomials on the interval $[0,z]$ take the form, for $x \in [0, z]$ and $j \in \NN_0$,
\begin{align}\label{def:Legendre}
e_{j}^z(x) = \sum_{k=0}^j a_{z,j,k} x^k,\qquad\qquad a_{z,j,k} = z^{-k-1/2}\sqrt{2j+1}(-1)^{j+k} \binom{j}{k}\binom{j+k}{k}.
\end{align}
Let $\langle \cdot,\cdot\rangle_{L^2[0,z]}$ denote the inner product in $L^2[0,z]$, so that $\langle h,e_j^z\rangle_{L^2[0,z]}$ are the 
coefficients in the expansion of $h\in L_2[0,z]$. By elementary computation (see \eqref{EqMeanQuadratic}), it can be seen that 
\begin{align}\label{eq:mu square int op}
\expt\langle \tau_{z,\q}(f),e_j^z\rangle_{L^2[0,z]} = \begin{cases}
\tfrac{1}{6}z^{5/2} + \tfrac{1}{2}z^{3/2}\q^2, & j = 0,\\
\tfrac{1}{12}\sqrt{3}z^{5/2} + \tfrac{1}{6}\sqrt{3}z^{3/2}\q^2, & j= 1,\\
\tfrac{1}{60}\sqrt{5}z^{5/2}, & j=2,\\
0, & j \geq 3.
\end{cases}
\end{align}
Since the expectations vanish after the third coordinate, we restrict the observational model to the
first three coefficients. Consider the model \eqref{eq: model} with $f_i$ i.i.d.\ Brownian motions and for given
$z_i>0$ and a given positive-definite matrix $\Lambda$, 
\begin{equation}
\label{EqMeanCovSquareIntegral}
T_{\q}^{\bi}(f_i)=\bigl(\langle \tau_{z_i,\q}(f_i),e_j^z\rangle_{L^2[0,z]}\bigr)_{j=0,1,2},\qquad\qquad
\Lambda^{\bi}=\Lambda.
\end{equation}
Thus the observational model can be written, for $G$ Wiener measure,
\begin{align}
\label{def:exam1:X}
X_i = T_{\q}^{\bi}(f_i)+\gamma_i,\qquad \gamma_i \iid N_3(0,\Lambda),\qquad f_i\iid G,\qquad i=1,\ldots, N.
\end{align}
In this hierarchical structure, the distribution of $X_i$ is a mixture of multivariate normal distributions,
and hence non-Gaussian. We consider the posterior distribution for the parameter
$\q$ relative to the Gaussian surrogate model with the true mean vector and covariance matrix,
given in \eqref{EqMeanCovarianceModel}. 

The following corollary gives the limiting Gaussian law of the posterior distribution.

\begin{corollary}\label{cor:examp1}
Consider the observational model \eqref{def:exam1:X} for a given sequence $(z_i)_{i\in\NN}$ of positive
numbers. Let $\Theta \subset (0,\infty)$ be endowed with a prior density $\pi$ 
that is continuous and strictly positive at the interior point $\q_{0}$. Then
\begin{enumerate}
\item \label{thm: BvM Gaussian Missp works} If the sequence $(z_i)_{i\in\NN}$ is bounded and satisfies $\limsup_{N \to \infty} N^{-1} \sum_{i = 1}^N z_i^3 > 0$, and the limit $\lim_{N\to\infty} V_{\q^*, N} =V_{*}$ exists, then the misspecified posterior distribution in Equation~\eqref{def:posterior} converges to the Gaussian limit as in Equation~\eqref{eq:BvM-convergence}.
\item If $z_i \ra 0$, then $\lim_{N\to\infty} V_{\q^*, N} = 0$.
\end{enumerate}
\end{corollary}

The proof of the corollary is based on Theorem~\ref{cor:BvM:hierarhical} and is deferred to Section~\ref{sec:cor:examp1}.
The limiting misspecified posterior covariance $V_{*}^{-1}$ does not necessarily match 
the covariance of the true posterior distribution.
We compare the true and misspecified posterior numerically for synthetic data sets in Section~\ref{sec:Simulations}.

\begin{remark}\label{remark:counter}
This model does not satisfy the condition imposed in \cite{Kleijn12} for obtaining contraction at $\sqrt{N}$-rate.
In that paper it is assumed that there exists a function $m^*$ with $P_{\q_0,i}(e^{s m^*(X_i)}) < \infty$ for some $s>0$, 
such that for every $\q_1,\q_2$ in some open neighbourhood of $\q^*$,
\begin{align}\label{eq:counter}
\Bigl|\log \frac{q_{\q_1,i}}{q_{\q_2,i}}(X_i)\Bigr|\leq m^*(X_i)\|\q_1-\q_2\|.
\end{align}
To see that this fails, for simplicity take $\q_{0} = 0$ (hence $\q^* = 0$ as well) and $d=1$, i.e.\ consider only the first coordinate of the three-dimensional observation.
Recalling that the 0th Legendre-polynomial is equal to $1$ on the whole unit interval, 
we get that $X_{i} = \int_0^1 \int_0^t f_i(s)^2\, ds\, dt + \gamma_{i}$, 
for $f_i$ a Brownian motion and $\gamma_i \sim N(0,\Lambda)$.
In view of  Equations~\eqref{eq:mu square int op} and \eqref{eq:sig square int op}, 
\begin{align*}
\log \frac{q_{\q_1,i}}{q_{\q_2,i}}(X_i)
&= \frac{1}{2} \log \Big(\frac{\sigma_{\q_2}^{\bi}}{\sigma_{\q_1}^{\bi}}\Big)^2 + \Big(\frac{X_i - \mu_{\q_2}^{\bi}}{\sigma_{\q_2}^{\bi}}\Big)^2 - \Big(\frac{X_i - \mu_{\q_1}^{\bi}}{\sigma_{\q_1}^{\bi}}\Big)^2,
\end{align*}
which is  a quadratic function in $X_i = \int_0^1 \int_0^t f_i(s)^2\, ds\, dt + \gamma_{i}$. Since
$\int_0^1 \int_0^t f_i(s)^2 \,ds\, dt$ is lower bounded by a multiple of a $\chi_1^2$ distributed random variable,
there exists no function  $m^*$ satisfying \eqref{eq:counter} such that $P_{\q_0,i} e^{s m^*(X_i)}$ is finite for some $s>0$.
\end{remark}


\subsection{Schrödinger equation}\label{sec:Schrodinger}
The time-independent Schrödinger equation is a simple PDE-constrained,
non-linear inverse problem, which can serve as a benchmark for testing methodology.
Several authors have focused on nonparametric recovery of the underlying potential function.
Minimax posterior contraction rates using multi-scale analysis were derived for Gaussian priors 
in \cite{nickl2020convergence,Monardetal2021} and uniform sequence priors in \cite{Nickl18}, while 
Bernstein-Von Mises results on linear functionals were derived in \cite{Nickl18,Monardetal2021}.  In \cite{Koers:Sz:Vaart:2023}
 adaptive posterior contraction rates were derived using a linearization technique.
Presently we focus on the hierarchical setting \eqref{eq: model} with random potential function $f$,
and on estimating a parameter in the boundary function.

For a bounded domain $\O \subset \RR^p$, a positive function $f: \O\to(0,\infty)$
and a function $g_{\q}: \partial\O\to\RR$ that is known up to a parameter $\q\in\RR^d$, 
consider the solution $u=u_{\q, f}: \O \to \RR$  solving the equation
\begin{equation}\label{eq:Schrodinger}
\Biggr\{\begin{aligned}
\Delta u - 2f u &= 0, \qquad && \text{on } \O,\\
u &= g_{\q},\qquad && \text{on } \partial \O.
\end{aligned}
\end{equation}
The existence of a unique solution $u_{\q,f}$ is guaranteed for a sufficiently smooth domain and sufficiently smooth
functions $f$ and $g_\q$ (e.g.\ $f\in C^s(\O)$ and  $g_{\q} \in C^{s+2}(\O)$ for some $s>0$
suffice, see Proposition~$25$ in \cite{Nickl18}. For an appropriate orthonormal basis $(e_j)_{j\in\NN}$ of $L^2(\O)$
and given $p\ge d$, we define 
\begin{equation}
\label{def:Schrodinger:inner}
T_{\q}(f)=\bigl(\langle u_{\q,f}, e_j \rangle_{L^2(\O)}\bigr)_{j=1,\ldots,p}\in\RR^p.
\end{equation}
With this choice the hierarchical model \eqref{eq: model} can be written in the form,
for a probability distribution $G$ on a set of smooth, positive functions,
\begin{align}\label{def:Schrodinger:obs}
X_i=T_{\q}(f_i)+\gamma_i,\qquad \gamma_i \iid N(0,\Lambda),\qquad f_i\iid G,\qquad i=1,\ldots, N.
\end{align}
The observations $X_1,\ldots,X_N$ are i.i.d., but the functions $u_{\q,f_i}$, and hence the mean functions of the observations,
are random. 

We consider the posterior distribution for the parameter
$\q$ relative to the Gaussian surrogate model with the true mean vector and covariance matrix,
as given in \eqref{EqMeanCovarianceModel}. 
The following corollary describes the limiting law of the corresponding misspecified posterior distribution.
The proof is based on Theorem~\ref{cor:BvM:hierarhical} and is deferred to Section~\ref{sec:cor:examp2}.

\begin{corollary}\label{cor:examp2}
Consider the observational model \eqref{def:Schrodinger:obs} and assume that the function $\q\mapsto g_{\q}(x)$ in Equation~\eqref{eq:Schrodinger} is twice continuously
differentiable in $\q$ for every $x\in \partial \O$, and the functions $g_{\q}$, $\nabla g_{\q}$ and $\nabla^2g_{\q}$ are uniformly bounded on $\partial \O$.
Furthermore, assume that the vectors $\frac{d}{d\q_j}\mu_\q\big|_{\q=\q_{0}}\in\RR^p$, $j=1,\ldots,d$ are linearly independent,
and that either $\E_fT_\q(f)\not=\E_fT_{\q_0}(f)$, for every $\q\not=\q_0$ and the parameter set $\Theta$ is compact or
$\inf_{\|\q-\q_0\|\ge\d}\|\E_fT_\q(f)-\E_fT_{\q_0}(f)\|>0$, for every $\d>0$.
Finally, assume that the prior density $\pi$ is continuous and strictly positive at the interior point $\q_{0}$ of $\Theta$.
Then the misspecified posterior distribution admits the Gaussian approximation given in Equation~\eqref{eq:BvM-convergence exp} with $\q^* = \q_{0}$ and $V_{*}$ given by Equation~\eqref{eq:Vstar}.
\end{corollary}

We investigate the non-asymptotic behaviour of the misspecified posterior distribution
in a numerical study with simulated data in Section~\ref{sec:Simulations}.  

Rather than the observational model \eqref{def:Schrodinger:obs} based on observing  $p$
continuous inner products of the (random) functions $u_{\q, f_i}$, one could consider 
observing this function  at a grid of points over the domain $\O$. In the usual
theoretical setting this grid $(x_l)$ would consist of $M\ra\infty$ points, and the discrete averages
$M^{-1}\sum^M_{l=1} u_{\q,f_i}(x_l)e_j(x_l)$ would closely approximate the continuous
inner products used in \eqref{def:Schrodinger:inner}. In a theoretical analysis we would
replace the continuous inner products by the discrete inner products, but retain the
Gaussian error structure in \eqref{def:Schrodinger:obs}. For the usual choices of $M$ and $N$,
similar conclusions can be expected. Of course, in both cases the restriction to $p$ inner products
entails a loss of information, but this is inherent to working with a simplified (misspecified) model.

\subsection{Parabolic PDE-constrained inverse problem}\label{sec:PDE}
In this section, we consider parabolic PDEs with different time horizons.
This model is a step towards the more complex PDE-constrained inverse problems considered in practice,
illustrating the observation of a system at different stages of its evolution.

The model generalizes the PDE considered in \cite{kekkonen2022consistency}, who
derives a minimax posterior contraction rate for the attenuation (or absorption/depletion/creation) coefficient ($c_\q$ in the following)
in the context of the heat equation ($a=I$ and $b=f=0$).  In the same model \cite{Koers:Sz:Vaart:2023} derives adaptive psoterior contractin rates based on a linearization argument.
Here we consider a hierarchical model with a random source function $f$ and focus on estimation of parameters 
specifying the boundary function and attenuation coefficient.

For a bounded domain $\O\subset\RR^d$, and functions $a: \O\to\RR_+^{d\times d}$, $b: \O\to\RR^d$, and $f, c_\q: \O\to(0,\infty)$, 
$g_\q: (0,t_{\max})\times\partial\O\to\RR$, let $u_{\q,f}$ be the solution to the elliptic PDE, for given $t_{\max}>0$,
\begin{equation}\label{eq:parabolic}
\left\{\begin{aligned}
-\frac{\partial u(t,x)}{\partial t} + \mathcal{A}_\q u(t,x) &= f(x), \quad &&\text{on } (0,t_{\max})\times \O,\\
u(t,x) &= g_{\q}(t,x), \quad &&\text{on } [(0,t_{\max})\times \partial \O]\cup [\{t_{\max}\}\times \overline{\O}],
\end{aligned}\right.
\end{equation}
where $\mathcal{A}_\q$ acts on the function $x\mapsto u(t,x)$ for fixed $t$, and is given by,
for $v \in C^2(\O)$ with gradient $Dv$ and second derivative matrix $D^2v$,
\begin{align*}
\mathcal{A}_\q v := -\frac{1}{2}\tr(a D^2v) - \langle b, D v\rangle + c_{\q} v.
\end{align*}
The diffusion function $a$ and transport function $b$ are assumed fixed, while the
functions $c_\q$ and $g_\q$ are assumed known up to
a parameter $\q\in\Theta\subset\RR^d$. The source function is $f$ is considered random and unobserved.
We assume that the regularity conditions given in \cite{Feehan15} are satisfied, so that the solution $u_{\q,f}$ exists and can 
be represented by a Feynman-Kac formula. In particular, all functions  $a, b, c_{\q}, f, g_\q$ are continuous.
Following this reference, the PDE is posed in backward form, with a boundary condition at $t_{\max}$.

Fix  an orthonormal basis  $(e_{j})_{j \in\N}$ for $L^2(\O)$ with $\sup_{j = 1, \ldots, p} \|e_j\|_{\infty} < \infty$.
For a given sequence $(t_i)_{i\in\N}$ of positive numbers, define
$$T_{\q}^i(f)=\bigl(\langle u_{\q,f}(t_i,\cdot), e_j \rangle_{L^2(\O)}\bigr)_{j=1,\ldots,p}\in\RR^p.$$
We then consider observing $X_1,\ldots,X_N$ given by \eqref{def:Schrodinger:obs}, 
where $G$ is taken to be a distribution on the space of nonnegative continuous functions such that $\expt_G\|f\|_\infty^4<\infty$. 

Thus our observational scheme follows the hierarchical model \eqref{eq: model}.
We form the misspecified posterior distribution \eqref{def:posterior} according to the surrogate likelihood given
in \eqref{EqMeanCovarianceModel}. The next corollary describes the limiting law of this posterior distribution.
The proof is deferred to Section~\ref{sec:cor:examp3}.

\begin{corollary}\label{cor:examp3}
Let $\q_{0}$ be an interior point of the compact parameter set $\Theta$ and assume that $(t_i)_{i\in\N}$ are positive numbers 
such that $V_{*} =\lim_{N\to\infty} N^{-1}\sum_{i=1}^N V_{*}^{\bi} $ exists and is positive definite, where $V_{*}^{\bi}$ is defined in Equation~\eqref{eq:Vstar}. 
Suppose that $\q\mapsto g_\q(t,x)$ and $\q\mapsto c_{\q}(x)$ are twice differentiable 
with equicontinuous second derivatives.
Assume that $(\mu_{\q}^{\bi}, \Sigma_{\q}^{\bi}) \neq (\mu_{\q_{0}}^{\bi}, \Sigma_{\q_{0}}^{\bi})$ for all $\q\neq\q_{0}$,
and that for all $\delta > 0$, 
\begin{equation*}
\limsup_{N \to \infty} \sup_{\q: \|\q - \q_0\| > \delta} \frac{1}{N} \sum_{i = 1}^N \|\mu_{\q}^{\bi} - \mu_{\q_0}^{\bi}\|^2 > 0.
\end{equation*}
Let $\pi$ be a prior density that is positive and continuous at $\q_{0}$.
Then the misspecified posterior distribution
admits the Gaussian approximation given in Equation~\eqref{eq:BvM-convergence exp} with $\q^*=\q_{0}$ and precision matrix $V_{*}$.
\end{corollary}


\section{Numerical analysis}\label{sec:Simulations}
In this section, we study the numerical accuracy of the Gaussian approximation relative to the misspecified and true posterior distributions, both in the toy model with the integral of the shifted squared Brownian motion and in the model involving the time-independent Schrödinger equation.

\subsection{Square integral operator}\label{subsubsec: square integral operator}
We consider the shifted, square integral operator $\tau_{1,\q}(f)=\int_0^{\cdot}\bigl(f(s)-\q\bigr)^2 \,ds$ 
given in Equation~\eqref{def:exam1:T}, where $f$ follows a Brownian motion, and for
simplicity we take $z_i=1$, for every $i=1, \ldots, N$.
The true observational model for the data $X_1, \ldots, X_N\in\RR^p$ is the hierarchical model \eqref{def:exam1:X},
with $G$ Wiener measure and $p=3$.
(We let the index of the coordinates of $X_i$ start at $j=0$, so that the $j$th coordinate corresponds to the
$j$th Legendre polynomial \eqref{def:Legendre}.)
We consider the misspecified model using the surrogate Gaussian model \eqref{EqMisspecifiedQ},
as discussed in Section~\ref{subsec:square integral operator}.

We start our analysis by comparing the true Fisher information $I_{\q_{0}}$, the inverse variance of the misspecified posterior distribution $V_{\q_0}$, 
the variance $J_{\q_0}$ of the score function of the misspecified model, and the inverse sandwich variance $J_{\q_0}/V_{\q_0}^2$. 
The quantity $V_{\q_0}=V_*$ is derived in Lemma~\ref{thm:locMinKLDiv} (where $\theta^*=\q_0$,
and all $V_{*}^{\bi}$ are equal, as the data are i.i.d.), with entries to the formula given in Equations~\eqref{eq:mu square int op} and \eqref{eq:sig square int op}.
The number $J_{\q_0}$ is the variance of the normal limiting distribution of the 
sequence $N^{-1/2}\nabla l^{(N)}(\theta_N^*)$ in Assumption~\ref{assum: regular likel} 
and can here be computed as the variance under the true (mixed) distribution \eqref{def:exam1:X} 
of the derivative of the log misspecified likelihood with respect to $\q$. 
This computation can be based on the analytical formulas and is straightforward (see Section~\ref{SectionSandwich}).

The computation of the true Fisher information is more involved. The analytical expression is
\begin{align}\label{eq:fisherInf}
I_{\q} = \expt\biggl( \frac{\int \frac{d}{d\q} p_{\q, f}(X) \,dG(f)}{\int p_{\q, f}(X) \,dG(f)}\biggr)^2,
\end{align}
where $G$ denotes the law of a Brownian motion, the outer expectation is with respect to the marginal distribution of $X$, 
and $p_{\q, f}$ denotes the conditional density of $X$ given $\q$ and $f$, i.e.
\begin{align*}
p_{\q, f}(x) = (2\pi)^{-p/2} (\det \Lambda)^{-1/2} 
e^{-\frac{1}{2}\sum_{i=0}^{p-1}\sum_{j=0}^{p-1}(x_i - \langle \tau_{1,\q} (f), e_i^1 \rangle )(\Lambda^{-1})_{ij} (x_j - \langle \tau_{1,\q} (f), e_j^1\rangle)}.
\end{align*}
We estimated the expectation in Equation~\eqref{eq:fisherInf} numerically using the Monte Carlo method, based on $10^6$ draws from the 
Brownian motion $f$.
For each draw, we approximated the forward map $\tau_{1,\q} (f)$ by computing the integral on a grid of size $100$ on $[0,1]$.
We approximated the inner integrals in Equation~\eqref{eq:fisherInf} at each data point $X$ by averaging the likelihoods $p_{\q, f}(x)$ and their derivatives over the $10^6$ draws from the Brownian motion, and next the outer expectation by averaging the resulting quotients over $10^6$ draws from $X$.

\begin{figure}
\centering
\includegraphics[width=5cm]{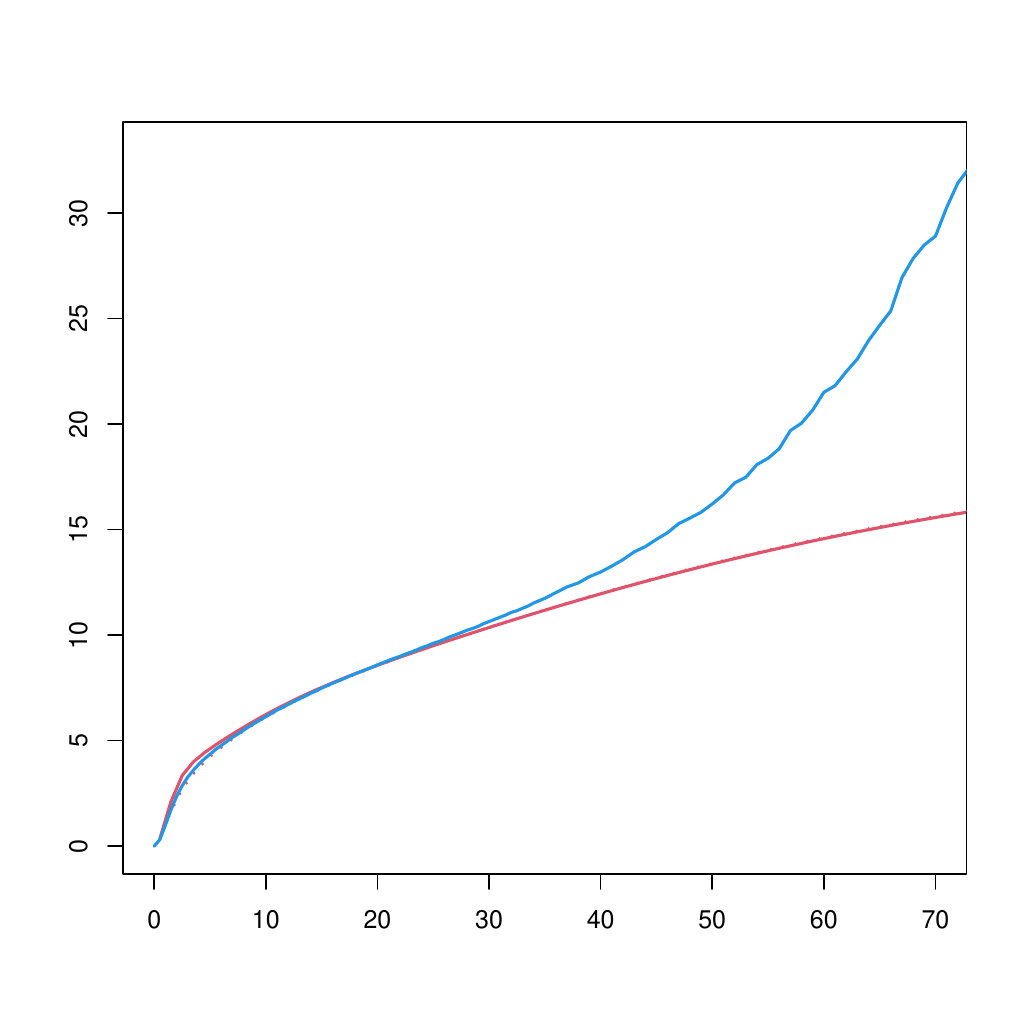}\includegraphics[width=5cm]{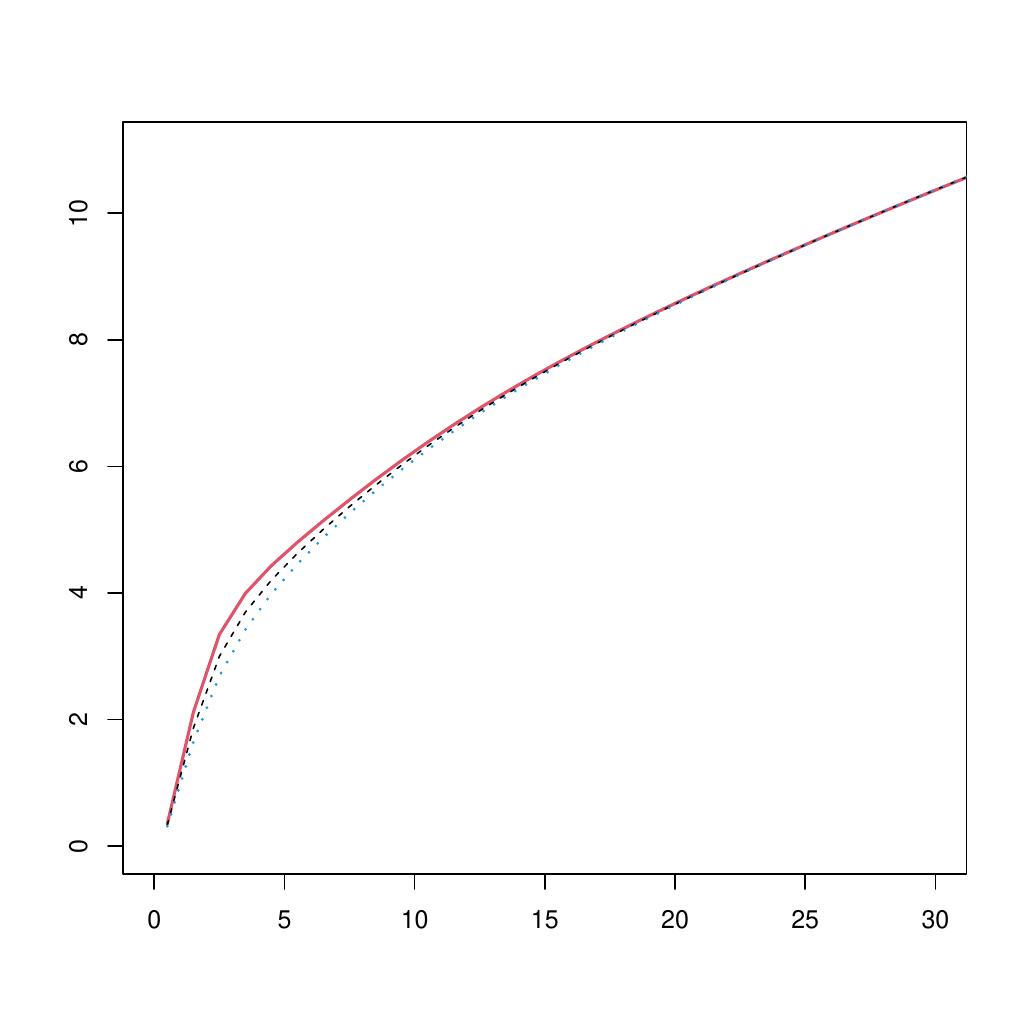}
\caption{Square integral observational model Equation~\eqref{def:exam1:X}, with $p=3$. The left panel shows the true  Fisher information $I_{\q_{0}}$ (blue) 
and the inverse variance $V_{\q_0}$ of the misspecified posterior (red) as a function of $\q_0$ on the horizontal axis. The right panel shows the inverse variance $V_{\q_0}$ of the misspecified posterior (red),
the variance $J_{\q_0}$ of the score function (black,dashed) and the inverse sandwich quantity $V_{\q_0}^2/J_{\q_0}$ (blue, dotted).}
\label{fig:toy:Fisher}
\end{figure}

The left panel of Figure~\ref{fig:toy:Fisher} shows that for small values of $\q_{0}$ the variance of the true and misspecified posterior distributions are aligned, 
but in the range $[40,70]$ the true Fisher information is larger, and hence the true posterior distribution will be more concentrated.
Inspection of the likelihood reveals that this difference disappears for extreme values of $\q_0\to\infty$, but in general this illustrates the superiority of using
the true posterior distribution, if feasible, but also that the relative efficiency depends on the true parameter value. The uncertainty quantification by the
true posterior distribution is accurate in view of the (well-specified) Bernstein-von Mises theorem. 
The right panel of Figure~\ref{fig:toy:Fisher} shows that in the present case the uncertainty quantification by the misspecified posterior distribution
is relatively accurate as well, as the deviation between the inverse sandwich value $V_{\q_0}^2/J_{\q_0}$ (the inverse variance of the
misspecified posterior mean) and  $V_{\q_0}$ (the inverse of the variance of the misspecified posterior distribution) is small. We restricted the range
of the plot to $\q_0\in [0,30]$ to highlight the small difference, mainly in the range $[3,10]$; for $\q_0>30$ the two curves are very close.

The preceding interpretations refer to the case of large sample size $n$, where the given quantities are theoretically guaranteed to give
accurate approximations. For further insight we compared the true and misspecified posterior distributions also for sample sizes
$N = 5, 50, 500$,  where we varied the shift parameter equal to $\q_{0}=5, 20, 70$. In all cases the prior density was
taken uniform. The results are shown in Figure~\ref{fig:density:small}, where the blue solid curve is the true posterior density and
the red solid curve the misspecified posterior density. These pictures confirm the large sample findings of the preceding paragraph.
The locations of the true and misspecified posterior density are similar, less so for $\theta_0=5$ 
where the true posterior density is better located, but the misspecified posterior also covers the true parameter. 
The spreads do not differ much in the cases $\q_0=5,20$, but more so in the case of $\q_0=70$, where the true posterior is more concentrated. 
The true value of the parameter is captured in a 95 \% credible interval in all cases, also for the misspecified
posterior density. The dashed curves in the figure show normal densities with means matching the true and 
misspecified posterior densities and variances equal to the asymptotic values $V_{\q_0}^{-1}$ and $I_{\q_0}^{-1}$, 
in corresponding colors. These are close in all cases, with the biggest discrepancies for the true posterior density when $N=5, 20$ 
and $\q_0=70$, when the asymptotic regime apparently kicks in only for larger sample sizes and the true posterior is less Gaussian.

We found the computational advantage of the misspecified posterior distribution to be substantial in this example. Not only could it
be computed in a fraction of the time needed for the true posterior, the computation also did not suffer from
the serious numerical instabilities inherent to computing the true posterior density by its form of a mixture over a high-dimensional
latent variable.

\begin{figure}
\centering
\includegraphics[width=5cm,height=3cm]{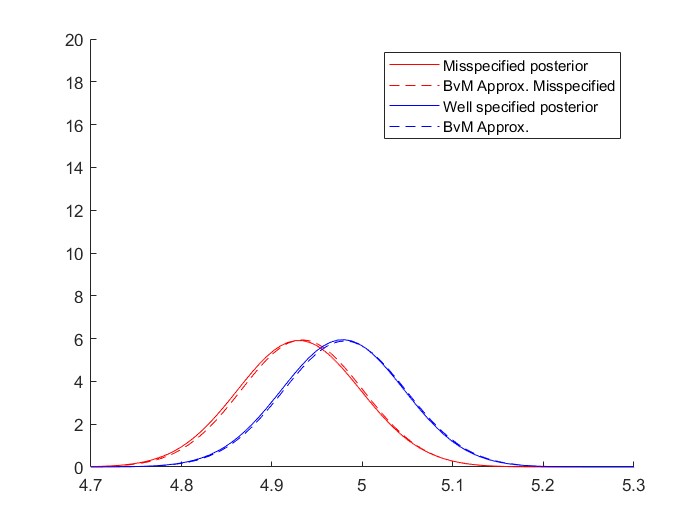}\hskip-4mm\includegraphics[width=5cm,height=3cm]{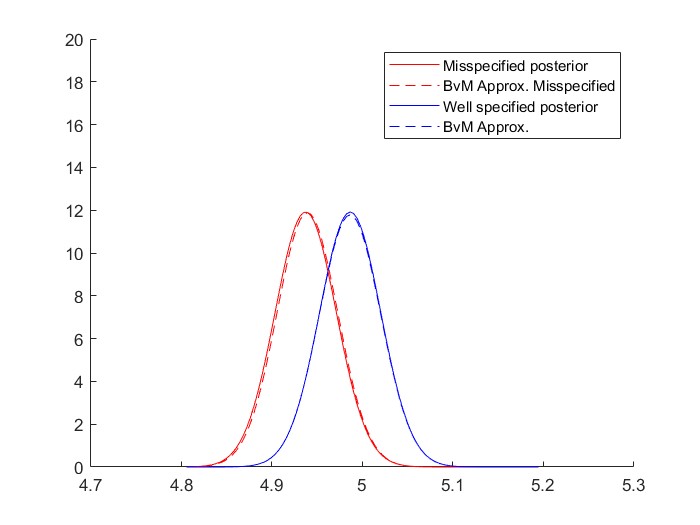}\hskip-4mm\includegraphics[width=5cm,height=3cm]{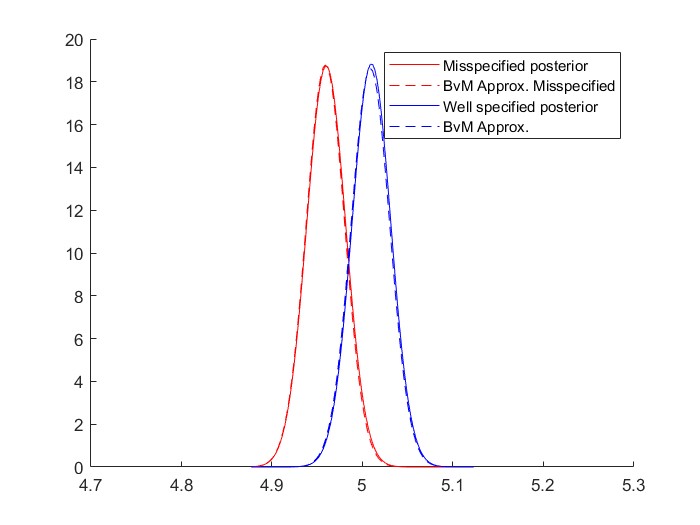}
\includegraphics[width=5cm,height=3cm]{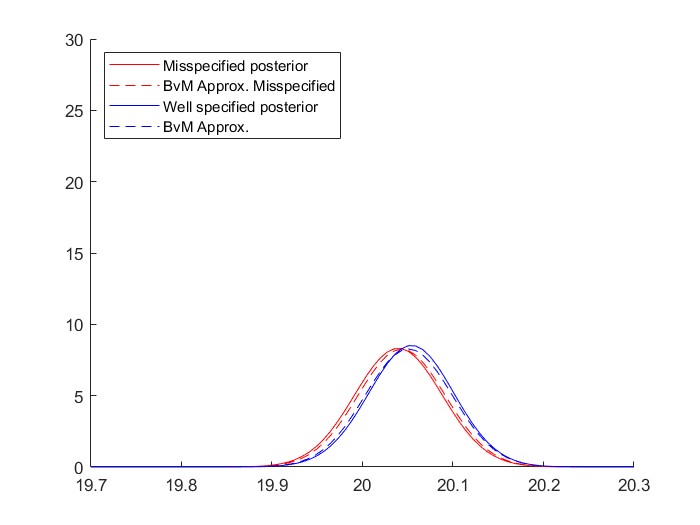}\hskip-4mm\includegraphics[width=5cm,height=3cm]{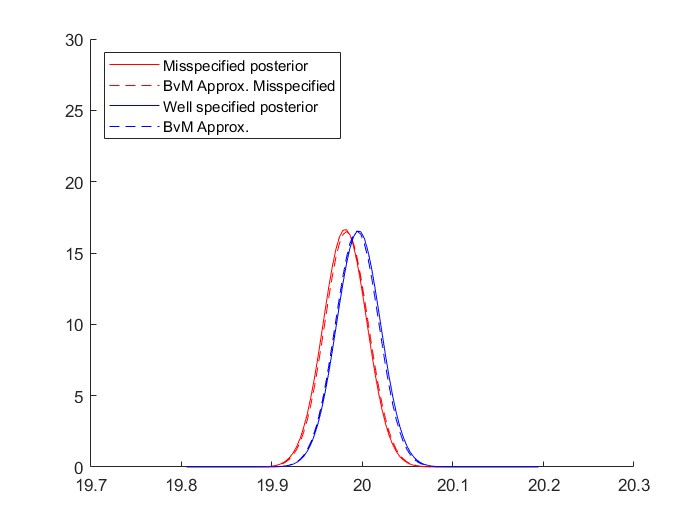}\hskip-4mm\includegraphics[width=5cm,height=3cm]{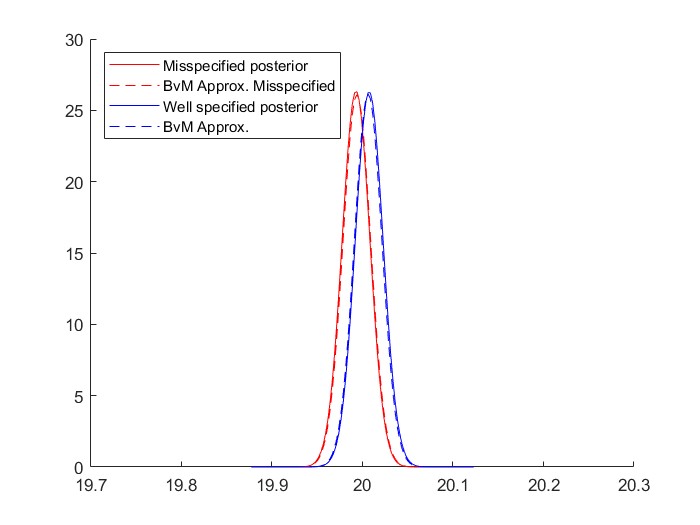}
\includegraphics[width=5cm,height=3cm]{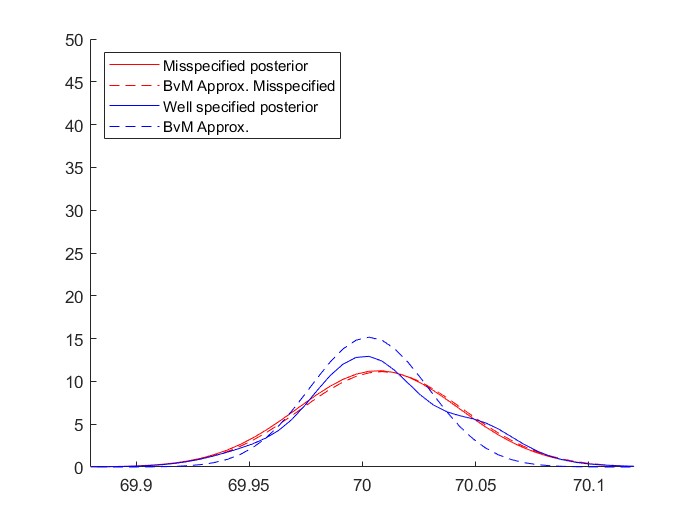}\hskip-4mm\includegraphics[width=5cm,height=3cm]{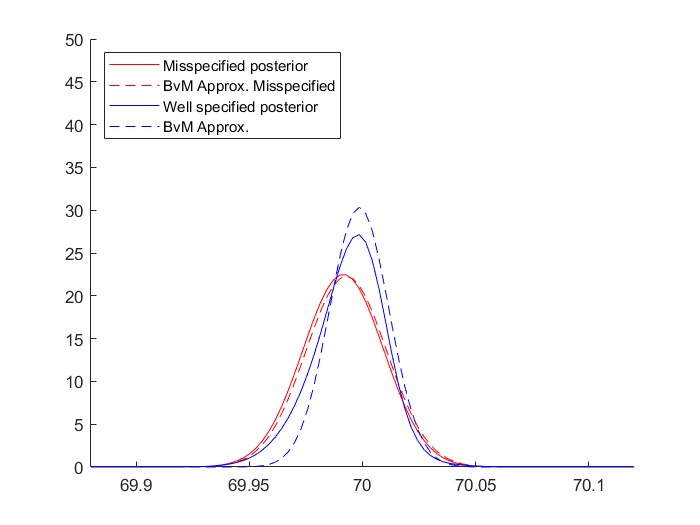}\hskip-4mm\includegraphics[width=5cm,height=3cm] {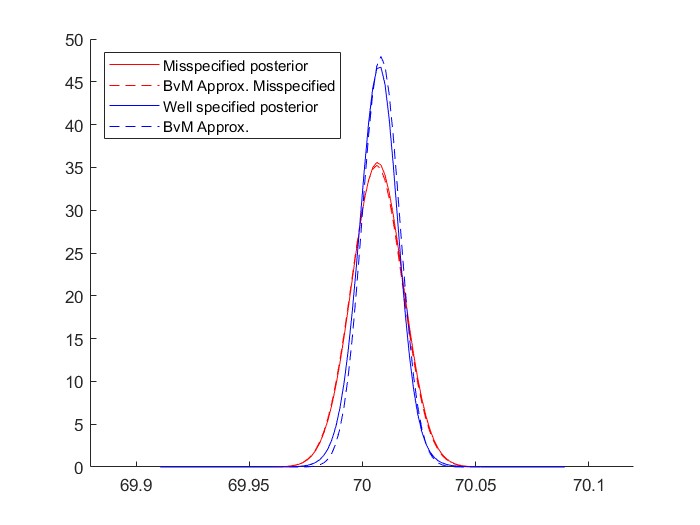}
\caption{Square integral observational model Equation~\eqref{def:exam1:X} with $p=3$. Each plot shows realisations for a single sample of observations of the true (blue) and misspecified (red) posterior densities and their corresponding limiting Gaussian approximations (dashed). From top to bottom the true parameter increases from $\q_0=5, 20$ to $70$, 
and from left to right the sample size increases from $N=50, 250$ to $1250$.}
\label{fig:density:small}
\end{figure}

\subsection{Schrödinger equation}
We consider the observational model Equation~\eqref{def:Schrodinger:obs}
 resulting from the time-independent Schrödinger equation~\eqref{eq:Schrodinger}, investigated in Section~\ref{sec:Schrodinger}.
We choose the domain $\O$ equal to the unit square, with the function $g_{\q}(x,y) = (x-\frac{1}{2})^2 + \q^2 y$ giving the boundary condition, 
for $(x,y) \in \partial \O$.  As the latent process in the first layer of the hierarchical model we take 
$f(x,y) = e^{2B_1(x) + 3B_2(y)}$, for $(x,y) \in [0,1]^2$, where $B_1, B_2$ are i.i.d.\ standard Brownian motions.
For the latent mean \eqref{def:Schrodinger:inner}, we use the basis functions $(x,y)\mapsto e_{i,j}(x,y):= e_i^1(x)e_j^1(y)$ for $0\leq i,j \leq 2$, 
where $e_i^1$ denotes the $i$th normalised Legendre basis on $[0,1]$ given in \eqref{def:Legendre} with $z=1$. Thus 
the dimension of the observations $X_i$ is $p=9$.

Computing the true posterior distribution was beyond our capacity. A straightforward mixed analytical and simulation approach as in
the preceding section would require an extreme number of calls of a PDE solver as well as cope with numerical instabilities.
An MCMC procedure would have to circumvent (or fill in) the latent Brownian bridge functions, two \emph{per observation}, and would
 pose a very substantial challenge, which is certainly not met by any of-the-shelve implementation, if feasible at all.
We therefore focus on computing the misspecified posterior distribution and on
comparing its variance to the variance of the posterior mean, given by the sandwich formula.

We computed the misspecified posterior density analytically by renormalising the product of the
misspecified Gaussian likelihood and the prior density, which was taken uniform. 
The mean vector and covariance matrix of the Gaussian likelihood
given in \eqref{EqMeanCovarianceModel} take the forms (with $\Lambda=I$)
\begin{align*}
\mu_\q&=\E _f T_\q(f)= \bigl(\E_f \langle u_{\q,f},e_{i,j}\rangle\bigr)_{i,j=0..2},\\
\Sigma_\q&=\bigl(\E_f\langle u_{\q,f},e_{i,j}\rangle\langle u_{\q,f},e_{i',j'}\rangle\bigr)_{i,j=0..2,i',j'=0..2}-\m_\q\m_\q^T+I.
\end{align*}
Here $u_{\q,f}$ is the solution to the Schr\"odinger Equation~\eqref{def:Schrodinger:obs} for given $f$ and $g_\q$, and
$\langle\cdot,\cdot\rangle$ is the inner product of $L_2\bigl([0,1]^2\bigr)$. We approximated the  expectations $\E_f$  
by averages over 500 samples of $f$. Because the boundary function $g_\q$ is a linear combination 
of the boundary functions $(x,y)\mapsto (x-1/2)^2$ and $(x,y)\mapsto y$ with coefficients $(1,\q^2)$, the solution
$u_{\q,f}$ is a linear combination of the solutions for these two boundary functions with the same coefficients.
This helped to speed up the computations for different values of $\q$, where it was sufficient to call the PDE solver
for a sample of functions $f$ combined with each of the two boundary functions. We used the standard finite elements PDE solver in MATLAB
({\tt https://it.mathworks.com/help/pde/ug/pde.pdemodel.solvepde.html}) to compute the solutions on a triangular mesh. 
The inner products $\langle\cdot,\cdot\rangle$ were computed by a simple trapezium rule on the same mesh.

Figure~\ref{FigureSchrodingerPosterior} shows plots of the misspecified posterior distribution for various sample
sizes and two values of $\q_0$, with a normal density of the same mean and of variance equal to the asymptotic value $V_{\q_0}^{-1}$.
Although the misspecified posterior is not normal, because the mean is not linear in $\q$ and the variance is not constant, the normal approximations
are very accurate, even at sample size 50, confirming the relevance of the asymptotic theory. 

To investigate the accuracy of the naive use of the credible sets defined by the misspecified posterior distribution, we computed
the inverse variance of the misspecified posterior distribution $V_{\q_0}$, 
the variance $J_{\q_0}$ of the score function of the misspecified model, and the inverse sandwich variance $J_{\q_0}/V_{\q_0}^2$. 
The computations follow Lemma~\ref{thm:locMinKLDiv} and Section~\ref{SectionSandwich}, similarly as in
Section~\ref{subsubsec: square integral operator}. 
Figure~\ref{FigureSchrodingerSandwich} shows that for small $\theta$ values (i.e. $\theta\leq 2$) these quantities are close to each other, while for higher $\theta$ the inverse sandwhich variance (slightly) dominates the inverse variance of the misspecified posterior distribution $V_{\q_0}$. These indicate,
that the uncertainty quantification of the misspecified posterior distribution is reasonably accurate (even if for larger $\theta$ values it is slightly conservative), at least for not too small sample sizes when the normal approximation is accurate.

Computing the mean vector and covariance matrix \eqref{EqMeanCovarianceModel}, which are required
to obtain the misspecified posterior distribution, is computationally intensive, as the PDE must be numerically solved
for a sample of latent processes $f$. However, the preceding shows the feasibility of this approach. In contrast,
computing the true posterior distribution seems infeasible, because of the hierarchical form of the model,
with an independent infinite-dimensional latent process attached to every observation. We performed
some numerical experiments to investigate a potential gain in precision from using the true posterior
distribution. This appeared modest at best, but we refreain from making precise claims in view of
the numerical instabilities in these experiments.

\begin{figure}
\centering
\includegraphics[width=5cm,height=3cm]{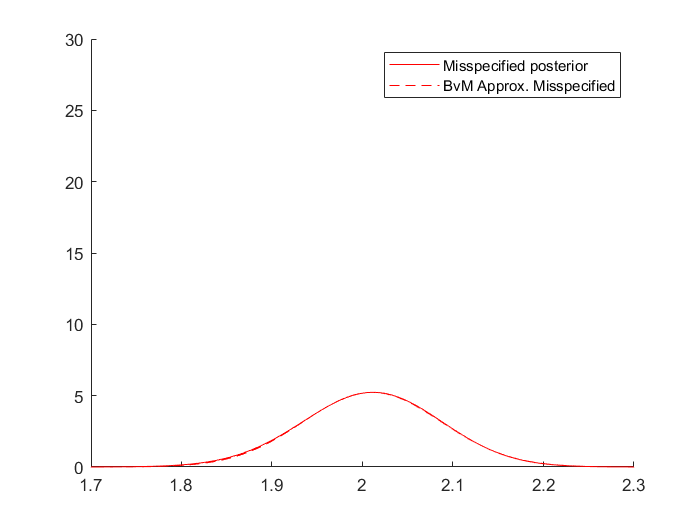}\hskip-4mm\includegraphics[width=5cm,height=3cm]{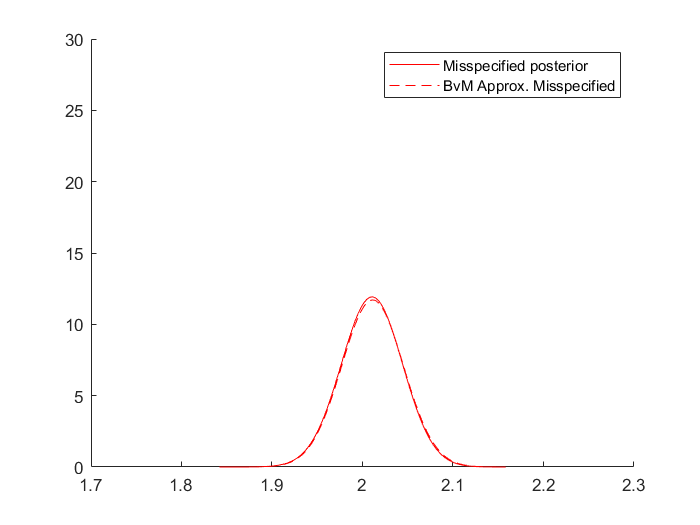}\hskip-4mm\includegraphics[width=5cm,height=3cm]{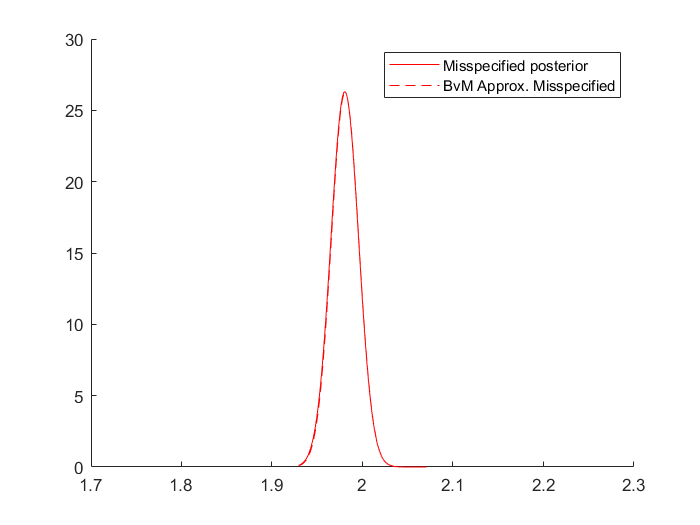}
\includegraphics[width=5cm,height=3cm]{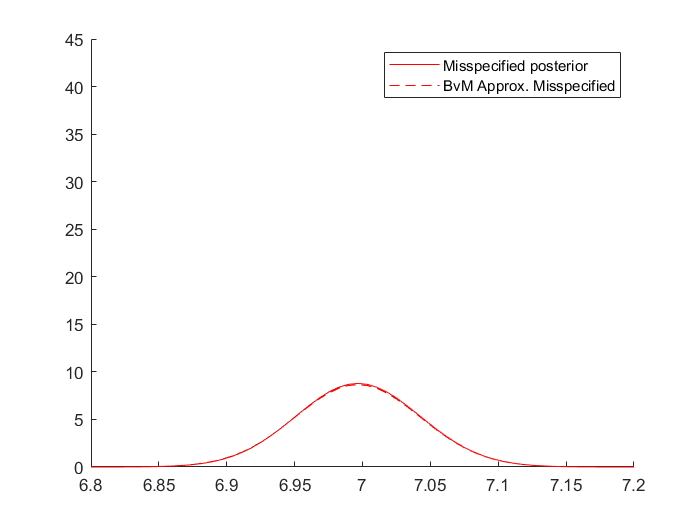}\hskip-4mm\includegraphics[width=5cm,height=3cm]{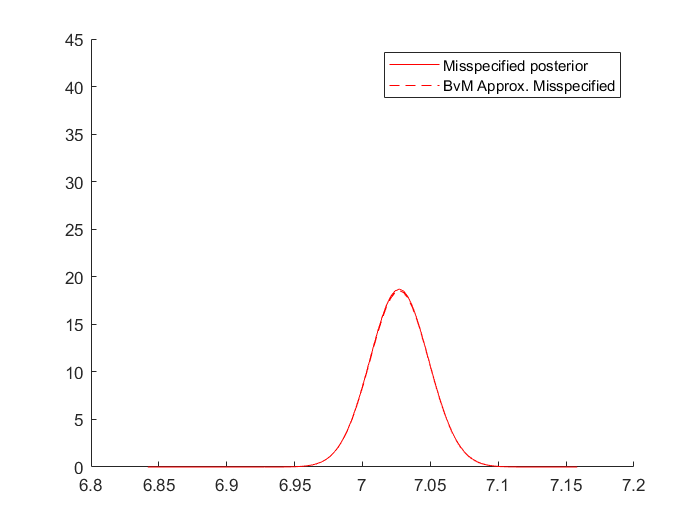}\hskip-4mm\includegraphics[width=5cm,height=3cm]{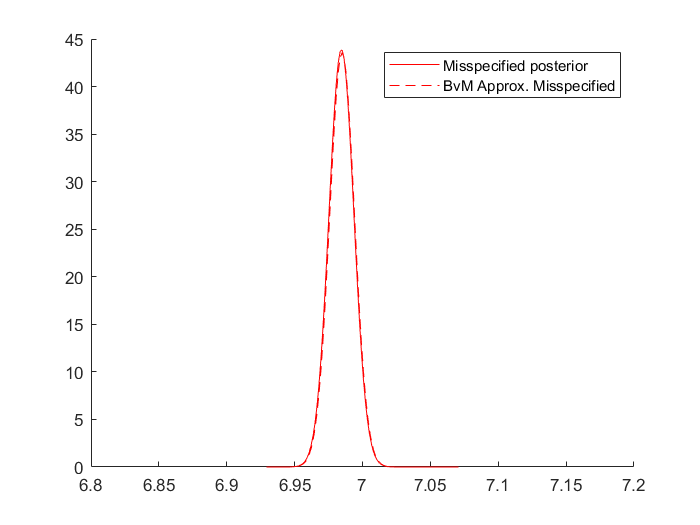}
\caption{Schr\"odinger equation observational model Equation~\eqref{def:Schrodinger:obs} with observations of dimension $p=3\times 3=9$. Each plot shows realisations for a single sample of observations of the misspecified (red) posterior density and its corresponding limiting Gaussian approximation (dashed, almost hidden). From top to bottom the true parameter increases from $\q_0=2$ to $\q_0=7$, 
and from left to right the sample size increases from $N=50, 200$ to $1250$.}
\label{FigureSchrodingerPosterior}
\end{figure}

\begin{figure}
\centering
\includegraphics[width=5cm,height=3cm]{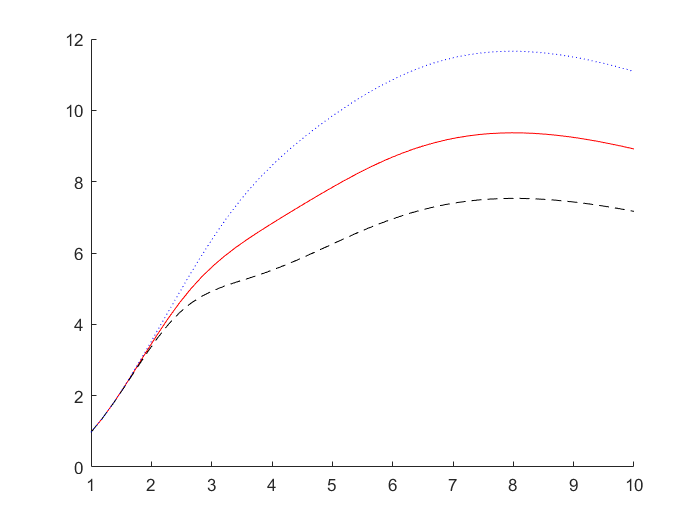}
\caption{Schr\"odinger equation observational model Equation~\eqref{def:Schrodinger:obs} with observations of dimension $p=3\times 3=9$. 
The inverse variance $V_{\q_0}$ of the misspecified posterior (red),
the variance $J_{\q_0}$ of the score function (black,dashed) and the inverse sandwich quantity $V_{\q_0}^2/J_{\q_0}$ (blue, dotted).}
\label{FigureSchrodingerSandwich}
\end{figure}

\section{Proof of the hierarchical Bernstein Von-Mises theorem}\label{sec:Proofs}
In this section, we give the proofs of Lemma~\ref{thm:locMinKLDiv} and Theorem~\ref{cor:BvM:hierarhical}.

\subsection{ Proof of Lemma~\ref{thm:locMinKLDiv}}\label{sec:prof:lem:variance}
For simplicity, we omit the index $i$ from the notation.

Instead of the normal distributions $Q_{\q,i}$ in \eqref{EqMisspecifiedQ} with the parametrised mean and covariance matrix, we shall consider
general normal distributions with free mean $\mu$ and covariance $\Sigma$, and show that 
the Kullback-Leibler divergence is uniquely minimised at $(\mu, \Sigma) = (\mu_{\q_{0}}, \Sigma_{\q_{0}})$.

Up to an additive constant minus twice the log-likelihood of the normal distribution with mean $\mu$ and covariance matrix $\Sigma$ is equal to $\log\det \Sigma+(X-\mu)\trans \Sigma^{-1}(X-\mu)$.
The expectation of this relative to an arbitrary distribution with $\expt X=\mu_{\q_{0}}$ and covariance matrix $\Cov X=\Sigma_{\q_{0}}$ is equal to
\begin{align*}
&\log \det\Sigma+ \expt \bigl((X - \mu_{\q_{0}})\trans \Sigma^{-1} (X - \mu_{\q_{0}}) 
+ (\mu-\mu_{\q_{0}})\trans \Sigma^{-1} (\mu-\mu_{\q_{0}}) \\
&\qquad\qquad\qquad=\log\det \Sigma+ \tr(\Sigma_{\q_{0}}\Sigma^{-1})+(\mu-\mu_{\q_{0}})\trans \Sigma^{-1}(\mu-\mu_{\q_{0}}).
\end{align*}
For every $\Sigma$, this is minimised at $\mu=\mu_{\q_{0}}$. For $B: = \Sigma_{\q_{0}}^{-1/2} \Sigma \Sigma_{\q_{0}}^{-1/2}$, its value at $\mu=\mu_{\q_{0}}$ can be written
$\log \det \Sigma_{\q_{0}} + \log \det B + \tr(B^{-1})= \sum_{j=1}^{p} (\log \lambda_j + \lambda_j^{-1})$,
for $(\lambda_j)_{j=1, \ldots, p}$ the eigenvalues of $B$.
This is minimised at $\lambda_j=1$, for $j=1,\ldots,p$. Equivalently, the minimiser in $B$
is the identity, corresponding to $\Sigma_{\q_{0}}$ as the minimiser with respect to $\Sigma$.

We now verify the formula for the Hessian at $\q = \q_{0}$.
The second derivative of the Kullback-Leibler divergence with respect to $\q_l$ and $\q_k$ is given by
\begin{align*}
&\frac{d^2}{d \q_l d \q_k} P_0 \log(p_0/p_\q)|_{\q=\q_{0}} = -\frac{1}{2}\det(\Sigma_{\q_{0}})^{-2}\Bigl(\frac{d}{d \q_l} \det(\Sigma_{\q_{0}})\Bigr)\Bigl(\frac{d}{d \q_k} \det(\Sigma_{\q_{0}})\Bigr) \\
&\quad+ \frac{1}{2}\det(\Sigma_{\q_{0}})^{-1}\frac{d^2}{d \q_l d \q_k} \det(\Sigma_{\q_{0}}) 
+ \Bigl(\frac{d}{d \q_l} \mu_{\q_{0}}\Bigr)\trans \Sigma_{\q_{0}}^{-1}\Bigl(\frac{d}{d \q_k} \mu_{\q_{0}}\Bigr) 
+ \frac{1}{2}\sum_{i,j} \Sigma_{\q_{0}, i, j}\frac{d^2}{d \q_l d \q_k} \Sigma_{\q_{0}, i, j}^{-1}.
\end{align*}
Combining this with Lemmas~\ref{lemma:JacobForm} and~\ref{lemma:secDerMat}, we see that
\begin{align*}
&\frac{d^2}{d\q_l d\q_k} P_0 \log(p_{0}/p_{\q})|_{\q = \q_{0}} - \Bigl(\frac{d}{d \q_l} \mu_{\q_{0}}\Bigr)\trans \Sigma_{\q_{0}}^{-1}\Bigl(\frac{d}{d \q_k} \mu_{\q_{0}}\Bigr)\\
&\quad= \frac{1}{2}\tr\Bigl(\Sigma_{\q_{0}}^{-1}\frac{d^2}{d\q_l d\q_k} \Sigma_{\q_{0}}\Bigr) - \frac{1}{2}\tr\Bigl(\Sigma_{\q_{0}}^{-1} \Bigl(\frac{d}{d\q_l}\Sigma_{\q_{0}}\Bigr) \Sigma_{\q_{0}}^{-1} \Bigl(\frac{d}{d\q_k} \Sigma_{\q_{0}}\Bigr) \Bigr) 
+ \frac{1}{2}\tr\Bigl(\Sigma_{\q_{0}} \frac{d^2}{d\q_l d\q_k} \Sigma_{\q_{0}}^{-1}\Bigr).
\end{align*}
By applying Equation~\eqref{eq:derivInvMat} twice, we obtain
\begin{align*}
\frac{d^2}{d\q_l d\q_k} \Sigma_{\q}^{-1}
&= \Sigma_{\q}^{-1} \Bigl(\frac{d}{d\q_l}\Sigma_{\q}\Bigr) \Sigma_{\q}^{-1} \Bigl(\frac{d}{d\q_k} \Sigma_{\q}\Bigr) \Sigma_{\q}^{-1} - \Sigma_{\q}^{-1} \Bigl(\frac{d^2}{d\q_l d\q_k}\Sigma_{\q}\Bigr) \Sigma_{\q}^{-1} \\
&\qquad\qquad\qquad
+ \Sigma_{\q}^{-1} \Bigl(\frac{d}{d\q_k} \Sigma_{\q}\Bigr) \Sigma_{\q}^{-1} \Bigl(\frac{d}{d\q_l}\Sigma_{\q}\Bigr) \Sigma_{\q}^{-1}.
\end{align*}
Plugging this expression into the formula in the second last display results in
\begin{align*}
&\frac{d^2}{d\q_l d\q_k} P_0 \log(p_{0} / p_{\q})|_{\q = \q_{0}}
= \frac{1}{2} \tr\Bigl( \Sigma_{\q_{0}}^{-1} \Bigl(\frac{d}{d\q_l} \Sigma_{\q_{0}}\Bigr) \Sigma_{\q_{0}}^{-1} \Bigl(\frac{d}{d\q_k} \Sigma_{\q_{0}}\Bigr)\Bigr)+ \Bigl(\frac{d}{d \q_l} \mu_{\q_{0}}\Bigr)\trans \Sigma_{\q_{0}}^{-1}\Bigl(\frac{d}{d \q_k} \mu_{\q_{0}}\Bigr).\nonumber
\end{align*}
This in turn implies the formula in Equation~\eqref{eq:Vstar} for $V_*$.

Finally note that Equation~\eqref{eq:Vstar} can be rewritten in the form
\begin{align*}
\frac{1}{2}\sum_{1\leq k,l \leq p} (A_{1, kl},\ \ldots,\ A_{p,kl})(A_{1, kl},\ \ldots,\ A_{p,kl})\trans + [v_1\trans , \ldots, v_p\trans ] [v_1\trans , \ldots, v_p\trans]\trans ,
\end{align*}
where $A_l=(A_{l,kl})_{k,l=1,\ldots,p}$.
This expression is the sum of Gram matrices and hence is positive semi-definite.
If the vectors $v_j$, $j=1,\ldots,d$ are linearly independent, then the second Gram matrix is positive definite, which implies the positive definiteness of $V_*$.

\subsection{Proof of Theorem~\ref{cor:BvM:hierarhical}}\label{sec:proof:cor}
Below we verify Assumptions~\ref{assum:BvM theorem}, with $\q_N^*=\q_{0}$ and
$V_{\q_N^*, N}$ equal to the Hessian of the Kullback-Leibler divergence.
Then the theorem is implied by Proposition~\ref{thm:BvM direct proof} and Lemma~\ref{thm:locMinKLDiv}.

\textit{Assumption~\ref{assum: convergence thetan}.} 
The point $\q_N^*=\q_{0}$, which minimises the 
Kullback-Leibler divergence between $Q_\q^{(N)}$ and $P_0^{(N)}$ by Lemma~\ref{thm:locMinKLDiv},
is independent of $N$ and is interior to $\Theta$ by assumption.

\textit{Assumption~\ref{assum: posit def V}.}
In case (a) of the theorem, where the mean and covariance functions do not depend on $i$, 
the average Hessian is fixed and given by Equation~\eqref{eq:Vstar}.
In case (b) the average Hessians converge by assumption.
In both cases, the (limiting) Hessian $V_{\q_{0}}$ is positive definite by assumption.

\textit{Assumption~\ref{assum: regular likel}.} 
Because the Kullback-Leibler divergence is minimised at $\q =\q_{0}$, its gradient vanishes at this point.
This implies that $\nabla l^{(N)}(\q_{0})$ has mean zero.
The random vector $\nabla l^{(N)}(\q_{0})$ is a sum of independent variables, each a polynomial of degree $2$ in the observations $X_{i,k}$, $k=1,\ldots,p$ .
Therefore $\var_{P_0}\bigl(N^{-1/2}\nabla l^{(N)}(\q_{0})\bigr)\lesssim N^{-1}\sum_{i=1}^N P_{0, i}\|X_i\|^4\lesssim 1$, by assumption.
Thus the sequence of variables $N^{-1/2}\nabla l^{(N)}(\q_{0})$ have mean zero and bounded  variances
and hence is bounded in probability.

\textit{Assumption~\ref{assum:likel max}.} 
For case (a) of the theorem we show below in Lemma~\ref{LemMLEBounded} that the sequence of induced maximum likelihood
estimators of the mean vector $\m_\q$ and inverse covariance $\Sigma_\q^{-1}$ are bounded in probability. 
In case (b) of the theorem the maximum likelihood estimators of $\q$ themselves are bounded in probability by assumption,
and then so are the mean vectors $\m_\q^i$ and inverse covariance matrices $(\Sigma_\q^i)^{-1}$ by their assumed
equi-continuity in $\q$.  Thus in both cases we can restrict the supremum in 
Assumption~\ref{assum:likel max} to a set $\Theta_1$ on which the means and inverse covariances are uniformly
bounded.

By the formula for the expected log Gaussian likelihood given in the proof of Lemma~\ref{thm:locMinKLDiv}, 
the expected value $2N^{-1}P_0\bigl(l^{(N)}(\q)-l^{(N)}(\q_{0})\bigr)$ is equal to minus the average given in 
Equation~\eqref{EqAsymptoticIdentifiability}.
In case (a) of the theorem these averages are fixed in $N$ and a continuous function of $(\m,\Sigma^{-1}):=(\m_\q,\Sigma_\q^{-1})$.
By assumption and construction, there exist $c,M>0$ such that the set $\{(\m_\q,\Sigma_\q^{-1}): \|\q-\q_0\|\ge\d, \q\in\Theta_1\}$ is contained
in the set $\{(\m,\Sigma^{-1}): c\le \|\m-\m_0\|\le M, c\le \|\Sigma^{-1}-\Sigma_0^{-1}\|\le M\}$, which is compact. The expression in 
Equation~\eqref{EqAsymptoticIdentifiability} is nonzero on this set and hence bounded away from zero.
It follows that the expected value $2N^{-1}P_0\bigl(l^{(N)}(\q)-l^{(N)}(\q_{0})\bigr)$ 
is bounded away from zero on $\{\q\in\Theta_1: \|\q-\q_0\|\ge\d\}$.
In case (b) of the theorem the set $\Theta_1$ from the preceding paragraph can be chosen compact 
and hence $\{\q\in\Theta_1: \|\q-\q_0\|\ge\d\}$ is contained in $\{\q\in\Theta: M\ge\|\q-\q_0\|\ge\d\}$, for some $M$.
In view of assumption~\eqref{EqAsymptoticIdentifiability}, again 
$2N^{-1}P_0\bigl(l^{(N)}(\q)-l^{(N)}(\q_{0})\bigr)$ is bounded away from zero on the set $\{\q\in\Theta_1: \|\q-\q_0\|\ge\d\}$.

Therefore, in both cases, for the verification of Assumption~\ref{assum:likel max}, it suffices to show that the sequence
$\sup_{\q\in\Theta_1} N^{-1} \bigl| \bigl( l^{(N)}(\q) - l^{(N)}(\q_{0})\bigr) - P_0 \bigl(l^{(N)}(\q) - l^{(N)}(\q_{0})\bigr) \bigr|$ tends to zero in probability.

We prove this using a bracketing argument, as given in Lemma~\ref{lemma:maximum likelihood non-iid}.
The Gaussian log-likelihood can be written in exponential family form as
\begin{align}\label{def:exp:fam}
l^{(N)}(\q)=c_N(\q)-\sum_{i=1}^N \sum_{k=1}^K T_{k}(X_i)g_{i,k}(\q) 
\end{align}
with sufficient statistics and natural parameter vectors given by
\begin{align}\label{eq:exp:fam:Gauss}
T_i = \begin{pmatrix} 
X_i\\ \vec\big( X_i X_i\trans \big) 
\end{pmatrix}\qquad\text{and} \qquad
g_i(\q) = \begin{pmatrix}
(\Sigma_\q^{\bi})^{-1}\mu_\q^{\bi}\\ -\frac{1}{2}\vec\big((\Sigma_\q^{\bi})^{-1}\big) 
\end{pmatrix}.
\end{align}
Here $\vec(A)$ denotes a column vector constructed from the elements of a matrix $A$ (in an arbitrary, but fixed way).
By assumption the functions $\q\mapsto (\Sigma_\q^{\bi})^{-1}$ and $\q\mapsto \mu_\q^{\bi}$ are twice differentiable with equicontinuous second derivatives.
The same holds for the functions $\q\mapsto g_{i,k}(\q)$, $i=1,\ldots, N$ and $k=1,\ldots,K$.

It suffices to prove that $\sup_\q N^{-1} \bigl| \sum_{i=1}^N \bigl( T_{i,k} - P_{0}T_{i,k}) \bigr) \bigl(g_{i,k}(\q) -g_{i,k}(\q_{0})\bigr)\bigr|$ tends to zero in probability, for every fixed $k$.
For a given covering $\Theta_1=\cup_j B_j$ by balls of radius $\delta$, consider the brackets
\begin{align*}
\mcall_{i,j} 
&= T_{i,k}^+\,\inf_{\q\in B_j}\bigl(g_{i,k}(\q)-g_{i,k}(\q_0)\bigr)-T_{i,k}^-\,\sup_{\q\in B_j}\bigl(g_{i,k}(\q)-g_{i,k}(\q_0)\bigr),\\
\mcalu_{i,j}
&= T_{i,k}^+\,\sup_{\q\in B_j}\bigl(g_{i,k}(\q)-g_{i,k}(\q_0)\bigr)-T_{i,k}^-\,\inf_{\q\in B_j}\bigl(g_{i,k}(\q)-g_{i,k}(\q_0)\bigr).
\end{align*}
We have $\mcall_{i,j} \leq T_{i,k} \bigl(g_{i,k}(\q) -g_{i,k}(\q_{0})\bigr)\leq \mcalu_{i,j}$, for all $\q \in B_j$ and $i \in \NN$, and
\begin{align*}
|\mcalu_{i,j} - \mcall_{i,j}| \leq |T_{i,k}|\, \Bigl|\sup_{\q\in B_j} g_{i,k}(\q) - \inf_{\q\in B_j} g_{i,k}(\q)\Bigr|.
\end{align*}
Because the functions $g_{i,k}$ are bounded on $\Theta_1$, for every given $\epsilon > 0$ there exists 
a finite cover $\Theta_1=\cup_j B_j$ such that variation of the vectors $g_{i,k}(\q)$ when $\q$ varies over one of the
$B_j$ is smaller than $\e$. 
For such a cover the right side of the display is bounded above by $|T_{i,k}|\epsilon$, for every $j$ and $i$.
It follows that $N^{-1}\sum_{i=1}^N P_0 |\mcalu_{i,j} - \mcall_{i,j}|
\le N^{-1}\sum_{i=1}^N P_0|T_{i,k}|\epsilon\lesssim \sup_i P_0\|X_i\|^4\epsilon$.
Furthermore, since $\var T^-\vee \var T^+\le \var T$ and $\var(S+T)\le 2\var S+2\var T$,
for any random variables $S, T$, we have
\begin{align*}
&\frac1{N^2}\sum_{i=1}^N \bigl(\var_{P_0} \mcall_{i,j}(X_i) + \var_{P_0}\mcalu_{i,j}(X_i) \bigr) 
\le\frac4{N^2} \sum_{i=1}^N \var_{P_0} T_{i,k} \max_j\sup_{\q \in B_j}\bigl|g_{i,k}(\q) -g_{i,k}(\q_{0})\bigr|^2,
\end{align*}
which tends to zero as $N \to \infty$.

\textit{Assumption~\ref{assum:Rn conv}.} A Taylor expansion of $l^{(N)}(\q)$ around $\q^*$ gives
\begin{align*}
&l^{(N)}(\q) - l^{(N)}(\q^*) 
= \nabla l^{(N)}(\q^*)\trans (\q - \q^*) + \frac{1}{2}(\q - \q^*)\trans \sum_{i=1}^N \sum_{k=1}^K H^{g_{i, k}}_{\q^*}(\q) T_{i, k} (\q - \q^*),
\end{align*}
where $H^{g_{i, k}}_{\q^*}(\q)$ are symmetric matrices with value at index $(r,s)$ given by
\begin{align}\label{eq:perturbed Hessian}
H^{g_{i, k}}_{\q^*, r, s}(\q) = 2\int_0^1 (1-t) \frac{\partial^2}{\partial \q_r \partial \q_s} g_{i, k}(\q^* + t(\q - \q^*)) \,dt.
\end{align}
Thus the Hessian of $N^{-1}l^{(N)}(\q)$ at $\q^*$ is given by
$-V_{\q^*, N} = \frac{1}{N}\sum_{i=1}^N \sum_{k=1}^K H_{\q^*}^{g_{i, k}}(\q^*) P_{0, i}(T_{i, k} )$,
and we can set 
\begin{align}\label{eq:Rn:exp}
\frac{1}{N}R_N(\q) = \frac{1}{N}\sum_{i=1}^N \biggl( \sum_{k=1}^K (H^{g_{i, k}}_{\q^*}(\q^*) P_{0, i} T_{i, k} - H^{g_{i, k}}_{\q^*}(\q) T_{i, k})\biggr).
\end{align}
By the triangle inequality this is bounded in absolute value above by
\begin{equation*}
\frac{1}{N} \Bigl\|\sum_{i=1}^N \sum_{k=1}^KH^{g_{i, k}}_{\q^*}(\q^*)(T_{i, k}- P_{0, i} T_{i, k})\Bigr\| 
+\max_{i,k} \sup_{\|\q - \q^*\| \leq \delta}\bigl|H^{g_{i, k}}_{\q^*}(\q^*)-H^{g_{i, k}}_{\q^*}(\q) \bigr|
\frac{1}{N} \sum_{i=1}^N \sum_{k=1}^KP_{0, i} |T_{i, k}|.
\end{equation*}
The first term is an average of independent random variables with mean zero and bounded variances, and hence this term tends to zero in probability.
The second term can be made arbitrarily small by choice of $\delta$, by the equicontinuity of the functions $H^{g_{i, k}}_{\q^*}$.

\textit{Assumption~\ref{assum:prior cont}.} This is true by assumption.

\begin{lemma}
\label{LemMLEBounded}
Under the conditions of Theorem~\ref{cor:BvM:hierarhical}(a), the vectors $\m_\q^\bi$ and matrices $\Sigma_\q^\bi$ and
 $(\Sigma_\q^\bi)^{-1}$ evaluated at the maximum likelihood estimator of $\q$ are bounded in probability.
\end{lemma}

\begin{proof}
Because $N^{-1}l^{(N)}(\q_0)$ is bounded in probability, it suffices to prove that for any given positive constants $C,\h$ 
there exist constants $K$ and $0<l\le L$ such that 
$$\liminf_{N\ra\infty}P_0^{(N)}\Bigl( \sup_{\q: \|\m_\q\|>K\text{ or } \l_{\min}(\Sigma_\q)<l\text{ or } \l_{\max}(\Sigma_\q)>L}\frac1N l^{(N)}(\q)< -C\Bigr)\ge 1-\h.$$
Here $\l_{\min}(\Sigma)$ and $\l_{\max}(\Sigma)$ denote the smallest and largest eigenvalues of a positive-definite
matrix $\Sigma$, and we drop the superscript $\bi$. The scaled log-likelihood can be written in the form
$$\frac1N l^{(N)}(\q)=-\frac12 \Bigl[\log \det \Sigma_\q+\tr\Bigl(\Sigma_\q^{-1}\bigl(S_N+(\bar X_N-\m_\q)(\bar X_N-\m_\q)^T\bigr)\Bigr)\Bigr].$$
for $\bar X_N$ the sample mean and $S_N=\frac 1N\sum_{i=1}^N (X_i-\bar X_N)(X_i-\bar X_N)^T$ the sample covariance matrix.
We shall show that with probability arbitrarily close to 1 the expression in square brackets becomes arbitrarily large
when $\|\m_\q\|>K$ or $\l_{\min}(\Sigma_\q)<l$ or $\l_{\max}(\Sigma_\q)>L$, for sufficiently large $K,L$ and small $l>0$.
For simplicity of notation, we drop the subscript $\q$ from $\Sigma_\q$ and $\m_\q$.

The expression in square brackets is bounded below by $\log \det \Sigma+\tr\bigl(\Sigma^{-1}S_N\bigr)$.
The sample covariance matrix converges almost surely to $\Sigma_{\q_0}$, which is invertible by assumption. It follows
that with high probability it is bounded below by a constant $c>0$ times the identity matrix. Hence the expression
in square brackets is with high probability bounded below by 
$\log \det \Sigma+c\tr\bigl(\Sigma^{-1}\bigr)=\sum_{j=1}^p\bigl[\log \l_j+c/\l_j\bigr]$, for $\l_j$ the eigenvalues
of $\Sigma$. Since $\inf_{\l>0}(\log \l+c/\l)\ge \log c+1$,
this is bounded below by $(p-1)(\log c+1)+\log \l_j+c/\l_j$, for any eigenvalue $\l_j$, and tends to infinity
both when $\l_j\ra0$ and when $\l_j\ra \infty$. 

The expression in square brackets is also bounded below by
$\log \det \Sigma+\tr\bigl(\Sigma^{-1}(\bar X_N-\m)(\bar X_N-\m)^T\bigr)$. 
For $0< l\le l_{\min}(\Sigma)\le\l_{\max}(\Sigma)\le L$, the matrix $\Sigma^{-1}$ is bounded
below by $\min_j \l_j^{-1}$ times the identity matrix, and the expression 
is further bounded below by 
$p\log l+L^{-1}\|\bar X_N-\m\|^2$.
If $\|\m\|\ge K$, then this is bounded below by $p\log l+L^{-1}\bigl(K^2-\|\m_0\|^2-O_P(1)\bigr)$.

The lemma follows upon combining the results of the preceding two paragraphs.
\end{proof}

\begin{lemma}\label{lemma:maximum likelihood non-iid}
For $i\in \NN$ and $\q\in\Theta$, let $f_{i,\q}: \RR^d \to \RR$ be a measurable function.
Assume that for every $\epsilon > 0$, there exist functions $\mcall_{i,j}, \mcalu_{i,j}: \RR^d \to \RR$, for $j=1, \ldots, J_{\epsilon}$ and $i \in \NN$ such that for every $\q \in \Theta$ there exists a $j \in \{1, \ldots, J_{\epsilon}\}$ such that $\mcall_{i,j} \leq f_{i, \q}\leq \mcalu_{i,j}$ for every $i \in \NN$, and such that 
\begin{align*}
\limsup_{N\to\infty} \frac{1}{N} \sum_{i=1}^N \expt \big(\mcalu_{i,j}(X_i) - \mcall_{i,j}(X_i)\big) &\leq \epsilon \\
\lim_{N\to\infty} N^{-2} \sum_{i=1}^N \frac{1}{N} \sum_{i=1}^N \bigl(\var \mcall_{i,j}(X_i) + \var\mcalu_{i,j}(X_i) \bigr) &= 0,
\end{align*}
for every $j \in \{1, \ldots, J_{\epsilon}\}$.
Then $N^{-1} \sup_{\q\in \Theta} \bigl|\sum_{i=1}^N \bigl(f_{i, \q}(X_i) - \expt f_{i, \q}(X_i)\bigr)\bigr| \to 0$ in probability.
\end{lemma}

\begin{proof}
Take arbitrary $\q \in \Theta$.
By assumption there exists a $j \in \{1, \ldots, J_{\epsilon}\}$ such that, for $N$ large enough,
\begin{align*}
\frac{1}{N}\sum_{i=1}^N \bigl[f_{i,\q}(X_i) - P_i f_{i,\q}\bigr]
&\leq \frac{1}{N}\sum_{i=1}^N \bigl(\mcalu_{i,j}(X_i) - P_i\mcalu_{i,j}\bigr) + 
\frac{1}{N}\sum_{i=1}^N P_i( \mcalu_{i,j} - \mcall_{i,j})\\
&\leq \frac{1}{N}\sum_{i=1}^N \bigl(\mcalu_{i,j}(X_i) - P_i\mcalu_{i,j}\bigr) + \epsilon.
\end{align*}
Hence the supremum over $\q$ of the left side is bounded above by
\begin{equation*}
\max_{j=1, \ldots, J_{\epsilon}} \frac{1}{N}\sum_{i=1}^N \bigl(\mcalu_{i,j}(X_i) - P_i\mcalu_{i,j}\bigr) + \epsilon.
\end{equation*}
By the assumption $\lim_{N\to\infty} N^{-2} \sum_{i=1}^N \var\mcalu_{i,j}(X_i) = 0$, this
tends in probability to $\epsilon$.
In combination with a similar argument using the lower brackets
$\mcall_{i,j}$, we find that the sequence given by 
$N^{-1} \sup_{\q\in \Theta} \bigl|\sum_{i=1}^N \bigl(f_{i,\q}(X_i)- \expt f_{i,\q}(X_i)\bigr)\bigr| $ 
is asymptotically bounded in probability by $\epsilon$.
This being true for every $\epsilon > 0$, gives the result.
\end{proof}

\section{Proofs for the applications}\label{sec:proof:examples}
In this section, we collect the proofs for the three examples discussed in Section~\ref{sec:examples}.
In each case we verify that the conditions of Theorem~\ref{cor:BvM:hierarhical} hold.

\subsection{Proof of Corollary~\ref{cor:examp1}}\label{sec:cor:examp1}
The mean function $\q\mapsto \mu_\q^{\bi}$ is given in Equation~\eqref{eq:mu square int op} with $z=z_i$,
while elementary (but cumbersome, see Section~\ref{sec:proof:square:int:var}) computations 
show that the covariance function can be written as, for  $j, j'\in\{0,1,2\}$,
\begin{equation}\label{eq:sig square int op}
\Sigma_{\q,j,j'}^{\bi} = \Lambda_{j, j'} + \sum_{k=0}^{j} \sum_{l=0}^{j'} a_{1,j,l} a_{1,j',k} (z_i^5b_{k, l} + \q^2 z_i^4 c_{k, l})
=:\Lambda_{j,j'} + A_{j,j'}z_i^5 + \q^2 z_i^4 B_{j,j'} ,
\end{equation}
where the numbers $a_{1,j,l}$ are the coefficients of the Legendre polynomials on $[0,1]$ and the numbers
$b_{k,l}$ and $c_{k,l}$ are given in \eqref{EqDefbc}. The matrices $A$ and $B$, defined by the preceding display, are
positive definite. Thus the matrices $\Sigma_\q$ are quadratic functions of $\q$, with $\Lambda\le \Sigma_\q\lesssim \Lambda+A+\q^2B$,
for bounded sets $\{z_i: i\in\NN\}$. Consequently, their eigenvalues are bounded away from zero and bounded above by a multiple of $1+\q^2$.

The limiting covariance matrix of the misspecified posterior, given in Equation~\eqref{eq:Vstar}, 
is completely determined by Equations~\eqref{eq:mu square int op} and \eqref{eq:sig square int op},
and can be computed to be (note that $\frac{d}{d\q} \mu_{\q}^\bi= \q z^{3/2}(1, \tfrac{1}{3}, 0, \ldots, 0)$)
\begin{align*}
V_*^{\bi} &= \frac{1}{2}\tr\biggl(\Bigl[\bigl(\Lambda + Az_{i}^5 + B\q_{0}^2 z_{i}^4\bigr)^{-1} 2\q_{0} B z_{i}^4\Bigr]^2\biggr) \\
&\qquad\qquad+ \q_{0} z_{i}^{3/2}(1, \tfrac{1}{3}, 0, \ldots, 0) \bigl(\Lambda + Az_{i}^5 + B \q_{0}^2 z_{i}^4\bigr)^{-1}\q_{0} z_{i}^{3/2}(1, \tfrac{1}{3}, 0, \ldots, 0)\trans.
\end{align*}
This can be written as $z_i^8 G_{\q_{0},1}(z_i) + z_i^3 G_{\q_{0},2}(z_i)$, for functions $G_{\q_{0}, 1}$ and $G_{\q_{0},2}$ 
that are bounded on compacta.
We immediately conclude that $V_*^{\bi} > 0$ if and only if $\q^*=\q_{0} \neq 0$.

Make the dependence of $\m_\q^i$ and $\Sigma_\q^i$ on $z_i$ explicit by writing them as
$\m_{\q,z_i}$ and $\Sigma_{\q,z_i}$. The maps $(\q, z) \mapsto \mu_{\q, z}$ and
$(\q, z) \mapsto \Sigma_{\q, z}$ are uniformly continuous.
Since $\Sigma_{\q, z}\geq \Lambda$, the same is true for  the function $(\q, z) \mapsto \Sigma_{\q, z}^{-1}$.
In the case where the sequence $(z_i)_{i\in\NN}$ is constant, the data are i.i.d.\ and therefore $\lim_{N\ra\infty} V_{\q^*, N} = V_* > 0$.
If $z_i\rightarrow0$, then $\lim_{z_i\ra0} V_*^{\bi} = 0$, and thus $\lim_{N\ra\infty} N^{-1}\sum_{i=1}^{N} V_*^{\bi} = 0$.

In order to verify \eqref{EqAsymptoticIdentifiability}, we first note that 
$\tr\bigl( \Sigma_{\q_{0}}^{\bi} \bigl( \Sigma_{\q}^{\bi} \bigr)^{-1} - I \bigr) - \log \det \bigl( \Sigma_{\q_{0}}^{\bi} \bigl( \Sigma_{\q}^{\bi} \bigr)^{-1}\bigr)$
is the Kullback-Leibler divergence  between two multivariate-normal distributions with covariances  $\Sigma_{\q}$ and $\Sigma_{\q_0}$ and hence is nonnegative.
Thus it suffices to show that
\begin{equation*}
\limsup_{N \to \infty} \sup_{\q: |\q - \q_0| \geq \delta} \frac{1}{N} \sum_{i = 1}^N (\mu_{\q}^{\bi} - \mu_{\q_0}^{\bi})\trans (\Sigma_{\q}^{\bi})^{-1} (\mu_{\q}^{\bi} - \mu_{\q_0}^{\bi}) > 0.
\end{equation*}
With $\lambda_1^{\bi}, \ldots, \lambda_d^{\bi}$ the eigenvalues of $\Sigma_{\q}^{\bi}$, the average in the left-hand side of the previous display is lower bounded by 
\begin{equation}\label{eq:lower bound EqAsymptoticIdentifiability}
\frac{1}{N} \sum_{i = 1}^N \min_{j = 1, \ldots, d} \frac1{\lambda_j^{\bi}} \|\mu_{\q}^{\bi} - \mu_{\q_0}^{\bi}\|^2\gtrsim 
\frac{1}{N} \sum_{i = 1}^N \frac1{1+\q^2}\frac{z_i^3(\q + \q_0)^2 (\q - \q_0)^2}{3},
\end{equation}
in view of Equation~\eqref{eq:mu square int op},  if the sequence $(z_i)_{i\in\NN}$ is bounded.
For $M\ge |\q - \q_0| \geq \delta$ this is bounded away from zero provided
$\limsup_{N \to \infty} N^{-1} \sum_{i = 1}^N z_i^3>0$, which is true by assumption.

Minus twice the scaled likelihood $N^{-1}l^{(N)}(\q)$ is equal to 
\begin{align*}
&\frac1N\sum_{i=1}^N\Bigl[\log\det\Sigma_\q^i+\tr\bigl((\Sigma_\q^i)^{-1}(X_i-\m_\q^i)(X_i-\m_\q^i)^T\bigr)\Bigr]
\ge \log \det\Lambda+\min_{j = 1, \ldots, d} \frac1{\lambda_j^{\bi}}\frac1N\sum_{i=1}^N\|X_i-\m_\q^i\|^2.
\end{align*}
Since $N^{-1}\sum_{i=1}^N\|X_i-\m_{\q_0}^i\|^2$ is bounded in probability, the right side
of the preceding display becomes arbitrarily large if $N^{-1}\sum_{i=1}^N\|\m_{\q_0}^i-\m_\q^i\|^2$ becomes arbitrarily large.
By \eqref{eq:lower bound EqAsymptoticIdentifiability} the latter is the case if $\|\q-\q_0\|\ra\infty$. This implies that 
the maximum likelihood estimator of $\q$ is bounded in probability.

We finish the proof by providing upper bounds for the fourth moments of the $X_i$.
By Jensen's inequality and Fubini's theorem, we have
\begin{align*}
P_{0}^{(N)} \langle \tau_{\q,z_i} f, e_{j}^{z_i}\rangle_{L^2[0,z]}^4
&\leq \int_0^{z_i} \int_0^t P_{0}^{(N)} (f(s) - \q)^8\, ds\, |e_{j}^{z_i}(t)|^4 \,dt< \infty,
\end{align*}
where in the last inequality we used that $f(s) \sim N(0,s)$, implying a finite eighth moment.

\subsection{Proof of Corollary ~\ref{cor:examp2}} 
\label{sec:cor:examp2}

By the Feynman-Kac formula (see Proposition~25 in \cite{Nickl18} or \cite{Freidlin}, Theorem~2.1 on page~127),
\begin{align}\label{eq: Feynman-Kac BvM}
u_{\q,f}(x) = \expt ^x \Bigl(g_{\q}(B_{\tau_{\O}}) e^{-\int_0^{\tau_{\O}} f(B_s) ds}\Bigr).
\end{align}
Here $B=(B_s: s\geq 0)$ denotes a $d$-dimensional Brownian motion, the superscript $x$ on
the expectation indicates that this has started at $x \in \O$, and $\tau_{\O}$ is its exit time from $\O$.
It is known that $x\mapsto \E^x\tau_\O$ is uniformly bounded (e.g.\ \cite{Nickl2023}, Lemma~A.4.1). For
a nonnegative function $f$ this gives the bounds 
$\inf_{y\in\partial \O} g_\q(y) \lesssim u_{\q,f}(x) \leq \sup_{y \in \partial \O} g_\q(y)$.
The assumed boundedness of $g_\q$ implies that the variable $X_{i,l}=\langle u_{\q,f},e_l\rangle_{L^2(\O)}+\g_{i,l}$ defined in 
\eqref{def:Schrodinger:obs} possesses a finite fourth moment.

Equation \eqref{def:Schrodinger:obs} and Fubini's theorem give
\begin{align*}
\m_{\q,l}&=\E_\q X_{i,l}=\int_{\O} \expt^x \Bigl( g_{\q}(B_{\tau_{\O}}) \expt_f e^{-\int_0^{\tau_{\O}} f(B_s) ds} \Bigr) e_l(x) \,dx,\\
\Sigma_{\q,l,l'}+\m_{\q,l}\m_{\q,l'}&=\E_\q X_{i,l}X_{i,l'}
=\int_{\O}\int_{\O} \expt^x \expt^{x'}\Bigl( g_{\q}(B_{\tau_{\O}}) g_\q(B'_{\tau_{\O}})\\
&\qquad\qquad\qquad\times\expt_f \bigl[e^{-\int_0^{\tau_{\O}} f(B_s) ds} e^{-\int_0^{\tau_{\O}'} f(B'_s) ds}\bigr]\Bigr) e_l(x)e_l(x') \,dx\,dx',
\end{align*}
where $B$ and $B'$ are independent Brownian motions, and the inner expectation $\expt_f$ 
is with respect to the random variable $f$ with law $G$ as in \eqref{def:Schrodinger:obs}. 
By differentiating under the integrals and expectations, we obtain similar expressions 
for the partial derivatives with respect to $\q$, to the same order as the map $\q \mapsto g_{\q}$ is (continuously) differentiable (and bounded).
In particular, the maps $\q \mapsto \mu_{\q}$ and $\q \mapsto \Sigma_\q$ are twice continuously differentiable.

The assumption that the vectors $\frac{d}{d\q_j}\mu_\q\in\RR^p$ are linearly independent,
for $j=1,\ldots,d$, implies that the vectors $v_j$, for $j=1,\ldots,d$,  in Equation~\eqref{eq:Vstar} are linearly independent.
Therefore, the second Gram matrix in Equation~\eqref{eq:Vstar} is positive definite, 
which implies the positive definiteness of $V_*$.

\subsection{Proof of Corollary~\ref{cor:examp3}} \label{sec:cor:examp3}
Define $X = (X_t)_{t\in[0,t_{\max}]}$ to be the solution to the stochastic differential equation defined in terms of $a,b$ and $c_\q$ 
and initial conditions $X_0 = x \in \O$ and let $\tau_\O$ be its exit time from $\O$ (see \cite{Feehan15}, Section~6), so
that the solution $u_{\q}: (0,t_{\max}) \times \O$ of Equation~\eqref{eq:parabolic} can be expressed by the 
Feynman-Kac formula as, with $\bar\t_\O=\t_\O\wedge t_{\max}$,
\begin{align}\label{eq: Feynman-Kac equation time-dependent}
u_{\q,f}(t,x) 
&= \expt ^{x} \Bigl[g_{\q}(\bar\tau_\O, X_{\bar\tau_\O})e^{-\int_t^{\bar\tau_{\O}} c_{\q}(X_s)\, ds}
+ \int_{t}^{\bar\tau_\O} f( X_{s})e^{-\int_t^{s} c_{\q}(X_v)\, dv}\,ds\Bigr].
\end{align}
Using this formula, we proceed by showing that the collections of functions $\{\q \mapsto \mu_{\q}^{\bi}: i \in \NN\}$ 
and $\{\q \mapsto \Sigma_{\q}^{\bi}: i \in \NN\}$ are equicontinuous.
Let $\q_1, \q_2 \in \Theta$ be given.
Note that
\begin{align*}
|\mu_{ \q_1, l}^{\bi} - \mu_{\q_2, l}^{\bi}|
&\leq \|e_l\|_{\infty} \expt _g \|u_{\q_1,f}(t_i, \cdot) - u_{\q_2,f}(t_i, \cdot)\|_{\infty}.
\end{align*}
The second factor on the right-hand side of the preceding display can be bounded from above by triangle inequality as 
\begin{align*}
\|u_{\q_1,f}(t_i, \cdot) - u_{\q_2,f}(t_i, \cdot)\|_{\infty}
&\leq \expt ^x |g_{\q_1}(\bar\tau_\O, X_{\bar\tau_\O})|
 \underbrace{\Big|e^{-\int_{t}^{\bar\tau_\O} c_{\q_1}(X_s)\, ds} - e^{-\int_{t}^{\bar\tau_\O} c_{\q_2}(X_s)\,ds}\Big|}_{=: \Delta_1}\\
&\qquad+ \expt ^x \Big|e^{-\int_{t}^{\bar\tau_\O} c_{\q_2}(X_s)\, ds}\Big|
\underbrace{|g_{\q_2}(\bar\tau_\O, X_{\bar\tau_\O}) - g_{\q_1}(\bar\tau_\O, X_{\bar\tau_\O})|}_{=: \Delta_2}\\
&\qquad+ \expt^x \! \int\limits_{t}^{\bar\tau_\O} |f(X_s)| \underbrace{\Big|e^{-\int_{t}^s c_{\q_1}(X_v)\,dv}
 - e^{-\int_{t}^s c_{\q_2}(X_v)\,dv}\Big|}_{=: \Delta_3}\ ds.
\end{align*}
In view of the bounds 
\begin{align*}
\Delta_1 &\leq t_{\max} \sup_{\q \in \Theta}\|\nabla_{\q} c_{\q}\|_{\infty}\|\q_1 - \q_2\|,\\
\Delta_2 &\leq \sup_{\q \in \Theta}\|\nabla_{\q} g_{\q}\|_{\infty}\|\q_1 - \q_2\|,\\
\Delta_3 &\leq t_{\max}\sup_{\q \in \Theta} \|\nabla_{\q} c_{\q}\|_{\infty} \|\q_1-\q_2\|,
\end{align*}
combined with the assumed bounds on $g_{\q}$ and $c_{\q}$, we find that
\begin{align}\label{eq: bound diff thetas}
&\|u_{\q_1,f}(t_i, \cdot) - u_{\q_2,f}(t_i, \cdot)\|_{\infty}
\leq t_{\max}\sup_{\q \in \Theta} \|g_{\q}\|_{\infty} \sup_{\q \in \Theta}\|\nabla_{\q} c_{\q}\|_{\infty}\|\q_1 - \q_2\|\nonumber \\
&\qquad\qquad+ \sup_{\q \in \Theta}\|\nabla_{\q} g_{\q}\|_{\infty}\|\q_1 - \q_2\| 
+ t_{\max}^2 \sup_{\q \in \Theta} \|f\|_{\infty} \sup_{\q \in \Theta} \|\nabla_{\q} c_{\q}\|_{\infty} \|\q_1 - \q_2\|.
\end{align}
We conclude that $\{\q \mapsto \mu_{\q,l,f}^{\bi}: i\in \N\}$ is equicontinuous.
Since the covariance term $\Sigma_{\q, l, k}^{\bi} = P_{0, i} X_{\q, i, l} X_{\q, i, k} - P_{0, i}X_{\q, i, l} P_{0, i} X_{\q, i, k}$, and the second term is equicontinuous, we only need to show that $\q \mapsto P_{0, i} X_{\q, i, k} X_{\q, i, l}$ is equicontinuous.
We have
\begin{align*}
&|P_{0, i} X_{\q_1, i, k} X_{\q_1, i, l} - P_{0, i} X_{\q_2, i, k} X_{\q_2, i, l}|\\
&\qquad\qquad\leq \|e_l\|_\infty \|e_k\|_{\infty} P_{0, i} \int_{\O} \int_{\O}\Big( |u_{\q_1,f}(t_i, x) - u_{\q_2,f}(t_i, x)| |u_{\q_1,f}(t_i, y)|\\
&\qquad\qquad\qquad\qquad\qquad+ |u_{\q_1,f}(t_i, y) - u_{\q_2,f}(t_i, y)| |u_{\q_2,f}(t_i, x)|\Big) \,dx \,dy.
\end{align*}
If we show that the first term of the integrand on the right-hand side is equicontinuous, the claim follows.
By Hölder's inequality, we bound the integral by
\begin{align*}
&\expt_G \int_{\O}\int_{\O}|u_{\q_1,f}(t_i, x) - u_{\q_2,f}(t_i, x)| |u_{\q_1,f}(t_i, y)|\,dx \,dy \\
&\qquad\qquad\qquad\leq \sqrt{\expt_G \int_{\O} |u_{\q_1,f}(t_i, x) - u_{\q_2,f}(t_i, x)| \,dx^2 \expt_G\Big( \int_{\O} |u_{\q_1,f}(t_i, y)| \,dy\Big)^2}.
\end{align*}
The second factor on the right-hand side is bounded by $\text{Vol}(\O)^2(\expt_G\|g_{\q_1}\|_{\infty} + t_{\max} \|f\|_{\infty})^2$.
The integrand in the first factor is bounded as
\begin{align*}
\expt_G|u_{\q_1,f}(t_i, x) - u_{\q_2,f}(t_i, x)| \leq C\|\q_1 - \q_2\|
\end{align*}
by Equation~\eqref{eq: bound diff thetas}, for a constant $C < \infty$ that does not depend on $i$.
We conclude that $\q \mapsto \Sigma_{\q, k, l}^{\bi}$ is also equicontinuous.
Similar reasoning shows that the map $(\q, t) \mapsto (\mu_{\q, t}, \Sigma_{\q, t})$ is continuous on $\Theta \times [0,t_{\max}]$ if $g_{\q}, f$ and $c_{\q}$ are continuous and bounded in $\q$.
This can be extended to differentiability in $\q$ up to the minimum order of differentiability of $g_{\q}$ and $c_{\q}$.

Similar to the proof in Section~\ref{sec:cor:examp1}, by the assumption on $\mu_{\q}$ it suffices to show that the eigenvalues of $\Sigma_{\q_0}^{\bi}$ are uniformly bounded in $i$.
We do this by showing that $|\eta_{\q_0}(f_i)|$ is uniformly bounded in $i$.
Due to Equation~\eqref{eq: Feynman-Kac equation time-dependent}, we can bound $|\eta_{\q_0}(f_i)|$ from above by 
\begin{equation*}
|\eta_{\q_0}(f_i)| \leq \sup_{j = 1, \ldots, p} \|e_j\|_{\infty} \biggl(\sup_{\substack{t \in (0, t_{\max}),\\ x \in \O}} g_{\q_0}(t, x) + t_{\max} \expt^{x} \|f\|_{\infty}\biggr) < \infty.
\end{equation*}
Since this is uniform in $i$, we conclude that the entries of $\Sigma_{\q_0}^{\bi}$ are uniformly bounded in $i$, and therefore its eigenvalues $\lambda_{1}^{\bi}, \ldots, \lambda_d^{\bi}$ are also uniformly bounded in $i$.

Finally, we note that $P_{0, i}\|X_i\|^4\lesssim \expt_G\|u(t_i,\cdot)\|_\infty^4+ P_{0, i}\|\gamma_i\|^4$.
The first term is bounded by a multiple of $t_{\max}^4\big(\|g_{\q}\|_\infty^4+\expt_G\|f\|^4_\infty\big)<\infty$, while the boundedness of the second term follows from the Gaussianity of $\gamma_i$.

\section{Technical lemmas}\label{sec:Technical Lemmas}
In this section, we collect technical lemmas used to prove our main results.
We start by recalling Jacobi's formula (see for instance \cite{Magnus19}), and extending this to the second derivative of the determinant.

\begin{lemma}[Jacobi's formula]\label{lemma:JacobForm}
Let $\Sigma: \Theta \to \RR^{n \times n}$ with $\Theta\subset\RR^d$ be a coordinate-wise differentiable map, such that $\Sigma_{\q}$ is invertible for each $\q \in \Theta$.
Then the partial derivative of the determinant for every $j=1, \ldots,d$ is given by
\begin{align*}
\frac{d}{d\q_j} \det \Sigma_{\q} = \det(\Sigma_{\q}) \tr\Bigl(\Sigma_{\q}^{-1} \frac{d}{d\q_j} \Sigma_{\q}\Bigr).
\end{align*}
\end{lemma}

\begin{lemma}\label{lemma:secDerMat}
Let $\Sigma : \Theta \to \RR^{n\times n}$ be as in Lemma~\ref{lemma:JacobForm}, where each coordinate is twice 
differentiable.
Then the second derivative of $\det(\Sigma_{\q})$ with respect to $\q_l$ and $\q_k$ is given by
\begin{align*}
\frac{d^2}{d\q_l d\q_k} \det(\Sigma_{\q})
&= \det(\Sigma_{\q}) \biggl(\tr\Big(\Sigma_{\q}^{-1}\frac{d}{d\q_l}\Sigma_{\q}\Big)\tr\Big(\Sigma_{\q}^{-1}\frac{d}{d\q_k}\Sigma_{\q}\Big)\\ 
&\qquad\quad\ - \tr\Bigl(\Sigma_{\q}^{-1}\Big(\frac{d}{d\q_l} \Sigma_{\q}\Big) \Sigma_{\q}^{-1}\frac{d}{d\q_k} \Sigma_{\q}\Bigr) + \tr\Big(\Sigma_{\q}^{-1} \frac{d^2}{d\q_l d\q_k} \Sigma_{\q}\Big)\biggr).
\end{align*}
\end{lemma}
\begin{proof}
Writing the first derivative with the help of Jacobi's formula and interchanging the order of trace and differentiation, we see
\begin{align*}
&\frac{d}{d \q_l d \q_k} \det(\Sigma_{\q})
= \frac{d}{d \q_l} \tr\Big(\det(\Sigma_{\q}) \Sigma_{\q}^{-1} \frac{d}{d \q_k} \Sigma_{\q}\Big)\\
&\qquad= \tr\Big(\Big(\frac{d}{d\q_l} \det(\Sigma_{\q})\Big) \Sigma_{\q}^{-1} \frac{d}{d\q_k}\Sigma_{\q} + \det(\Sigma_{\q}) \Big(\frac{d}{d\q_l} \Sigma_{\q}^{-1}\Big) \frac{d}{d\q_k} \Sigma_{\q} 
+ \det(\Sigma_{\q}) \Sigma_{\q}^{-1} \frac{d^2}{d\q_l d\q_k} \Sigma_{\q}\Big).
\end{align*}
Next, we apply again Jacobi's formula to the first term on the right-hand side and the formula
\begin{align}\label{eq:derivInvMat}
\frac{d}{d\q_l} \Sigma_{\q}^{-1} = -\Sigma_{\q}^{-1} \biggl(\frac{d}{d\q_l} \Sigma_{\q}\biggr) \Sigma_{\q}^{-1},
\end{align}
for the derivative of the inverse of $\Sigma_{\q}$, see for instance \cite{Magnus19}, in the second term on the right-hand side.
This results in the stated formula, concluding the proof.
\end{proof}

\section{Proof of Proposition~\ref{thm:BvM direct proof}}\label{sec:proof:BvM:direct}
The main component of the proof is Lemma~\ref{lemma:J1 conv} below, which is modelled on Lemma~7.1 in \cite{Lehmann83}.

For $\tilde{\Theta}_N = \sqrt{N}(\Theta - T_N)$ the rescaled parameter set, we write the posterior density as
\begin{align*}
\pi_N(t\given X^{(N)})
&= \frac{ \pi\bigl(T_N + \frac{t}{\sqrt{N}}\bigr)\exp\Bigl( l^{(N)}\bigl(T_N + \frac{t}{\sqrt{N}}\bigr)\Bigr)}{\int_{\tilde{\Theta}_N} \pi\bigl(T_N + \frac{u}{\sqrt{N}}\bigr) \exp\Bigl( l^{(N)}\bigl(T_N + \frac{u}{\sqrt{N}}\bigr)\Bigr) du}
= \pi\Bigl(T_N + \frac{t}{\sqrt{N}}\Bigr) e^{\omega_N(t)}\frac1{C_N},
\end{align*}
where
\begin{align}\label{def: omega t}
\omega_N(t) &= l^{(N)}\Bigl(T_N + \frac{t}{\sqrt{N}} - l^{(N)}(\q_N^*)\Bigr) - \frac{1}{2N} \nabla l^{(N)}(\q_N^*)\trans V_{\q_N^*, N}^{-1}\nabla l^{(N)}(\q_N^*),\\
C_N &= \int\limits_{\tilde{\Theta}_N} \pi\Bigl(T_N + \frac{u}{\sqrt{N}}\Bigr) e^{\omega_N(u)} \,du.
\end{align}
Since $\int_{\RR^d} e^{-\frac{1}{2}t\trans V_*t}\, dt = (2\pi)^{d/2} \mathrm{det}(V_*)^{-1/2}$ and the sets $\tilde{\Theta}_N $ grow to the full space by
Assumption~\ref{assum: convergence thetan}, Lemma~\ref{lemma:J1 conv} implies
\begin{align}\label{eq:C_N-convergence}
C_N \stackrel{P_{0}^{(N)}}{\to}\pi(\q^*) (2\pi)^{d/2} \mathrm{det}(V_*)^{-1/2} > 0.
\end{align}
It follows that the sequence $C_N^{-1}$ is bounded in probability.

We can rewrite Equation~\eqref{eq:BvM-convergence} as $C_N^{-1}$ times
\begin{align*}
&\int\limits_{\tilde{\Theta}_N} \biggl| \pi\Big(T_N + \frac{t}{\sqrt{N}}\Big)e^{\omega_N(t)} 
- C_N (2\pi)^{-d/2} \mathrm{det}(V_*)^{1/2} e^{-\frac{1}{2}t\trans V_* t}\biggr|\, dt\\
&\qquad \leq \int_{\tilde{\Theta}_N} \biggl| \pi(T_N + \frac{t}{\sqrt{N}})e^{\omega_N(t)} - e^{-\frac{1}{2}t\trans V_*t}\pi(\q^*) \biggr|\, dt\\
&\qquad\qquad\qquad\qquad+ \bigl| C_N (2\pi)^{-d/2} \mathrm{det}(V_*)^{1/2} - \pi(\q^*)\bigr| \int\limits_{\tilde{\Theta}_N} e^{-\frac{1}{2}t\trans V_*t} \, dt,
\end{align*}
by the triangle inequality.
Both terms on the right converge to zero in probability, as follows from Lemma~\ref{lemma:J1 conv} and Equation~\eqref{eq:C_N-convergence}, respectively.
This concludes the proof of Equation~\eqref{eq:BvM-convergence}.

For the proof of Equation~\eqref{eq:BvM-convergence exp}, it is enough to show that the preceding line of argument goes through with an added factor $(1+|t|^k)$ in the integrands.
This follows by an appropriately adapted version of Lemma~\ref{lemma:J1 conv}.

\begin{lemma}\label{lemma:J1 conv}
Under the assumptions of Proposition~\ref{thm:BvM direct proof}, as $N \to \infty$,
\begin{align*}
\int\limits_{\tilde{\Theta}_N} \biggl| e^{\omega_N(t)} \pi(T_N + \frac{t}{\sqrt{N}}) - e^{-\frac{1}{2}t\trans V_*t}\pi(\q^*) \biggr|\, dt \stackrel{P_{0}^{(N)}}{\to} 0.
\end{align*}
\end{lemma}

\begin{proof}
In view of the definitions of $R_N(\q)$ and $\omega_N(t)$, given in Equations~\eqref{eq:definition Rn} and \eqref{def: omega t}, 
\begin{align}\label{eq:w(t) rewrite}
\omega_N(t)
&= -\frac{1}{2}t\trans V_{\q_N^*, N}t - \frac{1}{2N} (t + Z_N)\trans R_N\Bigl(T_N + \frac{t}{\sqrt{N}}\Bigr) (t + Z_N),
\end{align}
where $Z_N := N^{-1/2}V_{\q_N^*,N}^{-1}\nabla l^{(N)}(\q_N^*)$.
By Assumptions~\ref{assum: posit def V} and~\ref{assum: regular likel}, 
the sequence $Z_N$ is bounded in probability.
Definition~\eqref{def:Tn} of $T_N$ gives that $T_N=\q_N^*+N^{-1/2}Z_N$ and hence
$T_N\to\q^*$ in probability, by Assumption~\ref{assum: convergence thetan}.
We now partition $\tilde{\Theta}_N$ into three subsets: for a fixed $M$ and $\eta$, we consider the sets of
$t\in\tilde\Theta_N$ with $\|t\| \leq M$, $M < \|t\| < \eta \sqrt{N}$ and $\|t\| \geq \eta \sqrt{N}$.

Because the volume of the ball with radius $M$ is finite, for the range $\|t\| \leq M$, it suffices to show that
\begin{equation}\label{eq:case1:Lemma:poster}
\sup_{\|t\| \leq M} \Bigl| e^{\omega_N(t)} \pi\Bigl(T_N + \frac{t}{\sqrt{N}}\Bigr) - e^{-\frac{1}{2}t\trans V_* t}\pi(\q_N^*) \Bigr|=o_P(1).
\end{equation}
In view of Assumption~\ref{assum:Rn conv}, we have that $\sup_{\|t\| \leq M} \bigl\| N^{-1} R_N\bigl(T_N + t/\sqrt{N}\bigr) \bigr\|\to0$ in probability.
It follows that the supremum over $\|t\|\le M$ of the second term on the right of Equation~\eqref{eq:w(t) rewrite} tends to zero in probability.
Combined with Assumption~\ref{assum: posit def V} this gives that $\sup_{\|t\| \leq M} |\omega_N(t)+\frac{1}{2}t\trans V_*\to 0$.
Since also $\sup_{\|t\| \leq M} |\pi(T_N + t/\sqrt{N}) - \pi(\q^*)| \to 0$ by Assumption~\ref{assum:prior cont}, 
it follows that Equation~\eqref{eq:case1:Lemma:poster} holds.
This is true for every fixed $M$.

Next, we deal with the range $M < \|t\| < \eta \sqrt{N}$.
The function defined by $t\mapsto \exp(-\frac{1}{2}t\trans V_*t)$ is integrable and can be made arbitrarily small by choosing large enough $M$.
We shall show that 
\begin{align}\label{eq:ref}
\int\limits_{M < \|t\| < \eta \sqrt{N}} \pi\Bigl(T_N +\frac t{\sqrt{N}}\Bigr)e^{\omega_N(t)}\,dt
\end{align}
can be made arbitrarily small by choosing sufficiently large $M$ and sufficiently small $\eta>0$.
On the range of the integral, we have $\|T_N + t/\sqrt{N} - \q^*\| < 2\eta$, with probability tending to one.
Therefore, by Assumption~\ref{assum:prior cont}, we have for sufficiently small $\eta>0$ that $\pi(T_N + t/\sqrt{N}) $ is bounded in probability.
By Assumption~\ref{assum:Rn conv} we can choose $\eta > 0$ still smaller, if necessary, so that
$\sup_{\|t\|<\eta\sqrt{N}} \bigl \|N^{-1} R_N(T_N + t/\sqrt{N}) \bigr\| < \tfrac{1}{4} \lambda_{\min}(V_*)$, with probability tending to one, where $\lambda_{\min}(V_*)>0$ denotes the smallest eigenvalue of $V_*$.
Assumption~\ref{assum: posit def V} gives that $\|V_{\q_N^*, N} - V_*\| < \frac{1}{4}\lambda_{\min}(V_*)$, for large enough $N$.
Then, in view of Equation~\eqref{eq:w(t) rewrite},
\begin{align*}
\omega_N(t)\le -\frac{1}{2}\lambda_{\min}(V_*)& \|t\|^2 + \frac{1}{2}\|V_{\q_N^*, N} - V_*\|\|t\|^2 + \frac{1}{4}\lambda_{\min}(V_*) \bigl(\|t\|^2 + \|Z_N\|^2\bigr)\\
&\leq -\frac{1}{8}\lambda_{\min}(V_*) \|t\|^2 + \frac{1}{4}\lambda_{\min}(V_*) \|Z_N\|^2,
\end{align*}
with probability tending to one.
The second term is bounded in probability, while the first term is quadratically decreasing, and hence dominates the expression.
We conclude that the expression in Equation~\eqref{eq:ref} can be made arbitrarily small by choosing sufficiently large $M$, for $\eta$ sufficiently small so that the preceding estimates hold.

Finally, we consider the range $\|t\| > \eta \sqrt{N}$.
As before, the contribution of the term involving $\exp(-\frac{1}{2}t\trans V_* t)$ tends to zero, and hence we need only deal with the term $ \pi(T_N +t/\sqrt{N})e^{\omega_N(t)}$.
In view of Assumption~\ref{assum:likel max}, there exists an $\epsilon > 0$ such that $\sup_{\|\q - \q_N^*\| > \eta} l^{(N)}(\q) - l^{(N)}(\q_N^*) \leq -\epsilon N$ with probability going to one.
On this event,
\begin{align*}
&\int\limits_{\|t\| > \eta \sqrt{N}} \Bigl| e^{\omega_N(t)} \pi\Bigl(T_N + \frac{t}{\sqrt{N}}\Bigr)\Bigr|\, dt \\
&\qquad\qquad \leq N^{d/2}e^{-N\epsilon - \frac{1}{2N}\nabla l^{(N)}(\q_N^*)\trans V_{\q_N^*, N}^{-1}\nabla l^{(N)} (\q_N^*)}\int\limits_{\|\q - T_N\| \geq \eta} \pi(\q) \,d\q.
\end{align*}
This tends to zero in probability.
This finishes the proof of the lemma.
\end{proof}

%
%

\begin{acks}[Acknowledgments]
We would like to thank Harry van Zanten for proposing the problem and the fruitful discussions in the beginning of the project.
Furthermore, we would also like to thank Elena Sellentin for discussing novel statistical problems in astronomy and helping us in selecting relevant examples.
\end{acks}
\begin{funding}
Co-funded by the European Union (ERC, BigBayesUQ, project number: 101041064).
Views and opinions expressed are however those of the author(s) only and do not necessarily reflect those of the European Union or the European Research Council.
Neither the European Union nor the granting authority can be held responsible for them.

The research leading to these results 
is partly financed by a NWO Spinoza prize by the Netherlands Organisation for Scientific Research (NWO).
\end{funding}

\begin{supplement}

\subsection{Proof of assertion \eqref{eq:sig square int op}}\label{sec:proof:square:int:var}
Since $f$ is a Brownian motion,
\begin{align*}
&\expt \langle \tau_{z,\q}(f), e_{i}^z \rangle \langle \tau_{z,\q} (f), e_{j}^z\rangle 
= \int_0^z \int_0^z \int_0^t \int_0^s \expt (f_v - \q)^2 (f_r - \q)^2\, dr \,dv\, e_i^z(t) e_j^z(s)\, ds\, dt.
\end{align*}
Elementary computations using the definition of Brownian motion give, that $\expt f_v^2 f_r^2= 3r^2 + (v-r)r$, 
for $v \geq r$, and $\expt f_r f_v^2=0=\expt f_r f_v^2$, and hence
\begin{align*}
\expt (f_v-\q)^2 (f_r-\q)^2
&= 3r^2 + (v-r)r + \q^2 v + 5\q^2 r  + \q^4.
\end{align*}
Substituting this in the right side of the preceding display, we find that this is equal to 
\begin{align}\label{eq:cross:prod}
&\int_0^z \int_0^z \int_0^t \int_0^s \Bigl[3 (r \wedge v)^2 + \big((r \vee v) - (r \wedge v)\big)(r \wedge v) \\
&\qquad \qquad\qquad+ \big(5(r \wedge v) + (r \vee v)\big)\q^2 + \q^4\Bigr]\, dr \,dv\ e_i^z(t) e_j^z(s)\,ds\,dt.\nonumber
\end{align}
The double inner integral, on $(r,v)$,  can be explicitly calculated to be
\begin{align}\label{eq:help:integrand}
&\frac{2}{3}(s \vee t) (s \wedge t)^3 -\frac{1}{3}(s \wedge t)^4+\frac{1}{4}(s \wedge t)^2 (s \vee t)^2 \\
&\qquad\qquad\qquad+ \q^2\Big( \frac{1}{2}(s \wedge t) (s \vee t)^2 + \frac{5}{2} (s \vee t) (s \wedge t)^2 - \frac{2}{3}(s \wedge t)^3 \Big)+\q^4st.\nonumber
\end{align}
Next, for general $k, \ell, m, n \in \NN$,
\begin{align*}
\int_0^z \int_0^z (s \wedge t)^k (s \vee t)^{\ell} s^m t^n \,ds\, dt 
&= \int_0^z \int_0^t s^{k + m} t^{\ell + n}\, ds\, dt + \int_0^z \int_{t}^z t^{k+n} s^{\ell + m} \,ds \,dt\\
&= \frac{(2+2k + m + n)z^{2+k+l+m+n}}{(1+k+m)(1+k+n)(2+k+l+m+n)}.
\end{align*}
Recalling the definitions of Legendre polynomials in Equation~\eqref{def:Legendre}, this allows to compute all integrals 
in Equation~\eqref{eq:cross:prod}, where the double inner integral can be rewritten in the form Equation~\eqref{eq:help:integrand}.
The coefficient of the $\q^4$ term is
\begin{align*}
\int_0^z \int_0^z st \,e_i^z(t) e_j^z(s)\,ds\, dt 
= \sum_{\ell = 0}^i \sum_{k = 0}^j a_{z,i,\ell} a_{z,j,k} \frac{z^{\ell + k + 4}}{(\ell+2)(k + 2)}. 
\end{align*}
Similarly, the coefficient of the $\q^2$-term can be seen to be equal to 
\begin{align*}
\sum_{\ell = 0}^i \sum_{k = 0}^j a_{z, i, \ell} a_{z, j, k} \biggl\{ \frac{1}{2}\frac{(4 + \ell + k) z^{5 + k + \ell}}{(2+\ell)(2+k)(5+\ell+k)}+ &\frac{5}{2} \frac{(6+\ell+k)z^{5+\ell+k}}{(3 + \ell)(3 + k)(5 + \ell + k)}\\
&\quad- \frac{2}{3} \frac{(8 + \ell + k) z^{5 + \ell + k}}{(4 + \ell)(4 + k)(5 + \ell + k)}\biggr\}.
\end{align*}
Finally, the constant term takes the form
\begin{align*}
\begin{aligned}
\sum_{\ell = 0}^i \sum_{k = 0}^j a_{z,i,\ell} a_{z,j,k} \biggl\{\frac{2}{3} \frac{(8 + \ell + k) z^{6 + \ell + k}}{(4 + \ell)(4 + k)(6 + \ell + k)} - \frac{1}{3} \frac{(10 + \ell + k) z^{6 + \ell + k}}{(5 + \ell)(5+ k)(6 + \ell + k)} 
+ \frac{1}{4} \frac{z^{6 + \ell + k}}{(3 + \ell)(3 + k)}\biggr\}.
\end{aligned}
\end{align*}
By similar computations we find that 
\begin{align}
\label{EqMeanQuadratic}
\expt \langle \tau_{z,\q}(f), e_i^z \rangle
= \sum_{\ell = 0}^i a_{z, i, \ell} \biggl\{\frac{1}{2}\frac{1}{\ell + 3}z^{\ell + 3} + \q^2 \frac{1}{\ell + 2}z^{\ell + 2}\biggr\}.
\end{align}
Combining the above displays, we see after simplifying the expressions, that the pseudo covariance
$\Cov(\langle \tau_{z,\q}(f), e_{i}^z \rangle ,\langle \tau_{z,\q} (f), e_{j}^z\rangle)$ is equal to 
\newcommand\smpl{\mkern-1.45mu + \mkern-1.45mu}
\begin{align}
&\sum_{\ell = 0}^i  \sum_{k = 0}^j a_{1, i, l} a_{1, j, k}\Bigl(  z^{5}\Bigl[\frac{2(8 \smpl \ell \smpl k) }{3(4 \smpl \ell)(4 \smpl k)(6 \smpl \ell \smpl k)} \! - \! \frac{10 \smpl \ell \smpl k}{3(5 \smpl \ell)(5+ k)(6 \smpl \ell \smpl k)} \! \Bigr]\label{EqDefbc}\\
&\qquad+\q^2 z^{4}\Bigl[\frac{2(6+k+\ell)}{(3+k)(3+\ell)(5+k+\ell)} - \frac{2(8+k+\ell)}{3(4+k)(4+\ell)(5+k+\ell)} \Bigr]\Bigr).\nonumber
\end{align}
This concludes the proof of the claim.

\subsection{Sandwich variance}
\label{SectionSandwich}
The covariance matrix of the misspecified posterior mean is given by the sandwich formula $V_*^{-1}J_{\q_0} V_*^{-1}$, where $V_*$ is the
covariance matrix of the misspecified posterior distribution and $J_{\q_0}$ 
is the covariance matrix of the normal limiting distribution of the sequence $N^{-1/2}\nabla l^{(N)}(\theta_N^*)$ 
in Assumption~\ref{assum: regular likel}. For our hierarchical setting the matrix $V_*=V_{\q_0}$ is given in Lemma~\ref{thm:locMinKLDiv}.
In the i.i.d.\ case for a one-dimensional parameter, the number $J_\q$ can be computed as follows. 

The score function of the misspecified normal  density  is equal to 
$$\frac{\partial}{\partial \q}\Bigl(-\frac12\log \det \Sigma_\q -\frac 12(X-\m_\q)^T\Sigma_\q^{-1}(X-\m_\q)\Bigr).$$
The number $J_\q$ is the variance of this variable for $X\sim X_i$ following the true distribution \eqref{eq: model},
i.e.\ $X\sim T_\q(f)+\e$, for independent variables $f\sim G$ and $\e\sim N(0,\Lambda)$. For $\Lambda=I$ this readily gives
\begin{align*}
J_\q
&=\var\Bigl(\dot \m_\q^T\Sigma_\q^{-1}(X-\m_\q)+\frac 12 (X-\m_\q)^T\Sigma_\q^{-1}\dot\Sigma_\q\Sigma_\q^{-1}(X-\m_\q)\Bigr)\\
&=\var_f\Bigl(\dot\m_\q^T\Sigma_\q^{-1}\bigl(T_\q(f)-\m_\q\bigr)+\bigl(T_\q(f)-\m_\q\bigr)^TD_\q\bigl(T_\q(f)-\m_\q\bigr)\Bigr)\\
&\qquad\qquad+4\E_f \bigl\|D_\q\bigl(T_\q(f)-\m_\q\bigr)\bigr\|^2+\|\Sigma_\q^{-1}\dot\m_\q\bigr\|^2+2 \tr(D_\q^2),
\end{align*}
where $D_\q=\frac12 \Sigma_\q^{-1}\dot\Sigma_\q\Sigma_\q^{-1}$. The remaining variance and expectation are computed
under $f\sim G$.
\end{supplement}


\bibliographystyle{imsart-nameyear} 
\bibliography{bibliography}       


\end{document}